\documentclass[reqno,11pt]{amsart}

\usepackage{a4wide}
\usepackage[english]{babel}
\usepackage[utf8]{inputenc}
\usepackage{amsmath,amssymb}
\usepackage{dsfont}
\usepackage{tikz}
\usepackage[hidelinks]{hyperref}

\theoremstyle{definition}
\newtheorem{thm}{Theorem}[section]
\newtheorem{lem}[thm]{Lemma}
\newtheorem{prop}[thm]{Proposition}

\newtheorem{defi}[thm]{Definition}

\numberwithin{equation}{section}

\newcommand{\Sc}{\operatorname{Sc}}
\renewcommand{\Im}{\operatorname{Im}}
\newcommand{\sgn}{\operatorname{sgn}}
\newcommand{\ran}{\operatorname{ran}}
\newcommand{\dom}{\operatorname{dom}}

\title[Quadratic estimates for the $H^\infty$-functional calculus]{Quadratic estimates for the $H^\infty$-functional calculus \\ of bisectorial Clifford operators}

\author[Fabrizio Colombo]{F. Colombo}
\address{(FC) Politecnico di Milano, Dipartimento di Matematica, Via E. Bonardi 9, 20133 Milano, Italy}
\email{fabrizio.colombo@polimi.it}

\author[Francesco Mantovani]{F. Mantovani}
\address{(FM) Politecnico di Milano, Dipartimento di Matematica, Via E. Bonardi 9, 20133 Milano, Italy}
\email{francesco.mantovani@polimi.it}

\author[P. Schlosser]{P. Schlosser}
\address{(PS) Graz University of Technology, Institute of Applied Mathematics, Steyrergasse 30, 8010 Graz, Austria}
\email{pschlosser@math.tugraz.at}

\begin{document}

\begin{abstract}
The $H^\infty$-functional calculus is a two-step procedure, introduced by A. McIntosh, that allows the definition of functions of sectorial operators in Banach spaces. It plays a crucial role in the spectral theory of differential operators, as well as in their applications to evolution equations and various other fields of science. An extension of the $H^\infty$-functional calculus also exists in the hypercomplex setting, where it is based on the notion of $S$-spectrum. Originally this was done for sectorial quaternionic operators, but then also generalized all the way to bisectorial fully Clifford operators. In the latter setting and in Hilbert spaces, this paper now characterizes the boundedness of the $H^\infty$-functional calculus through certain quadratic estimates. Due to substantial differences in the definitions of the $S$-spectrum and the $S$-resolvent operators, the proofs of quadratic estimates in this setting face additional challenges compared to the classical theory of complex operators.
\end{abstract}

\maketitle

AMS Classification 47A10, 47A60. \medskip

 \medskip

\textbf{Acknowledgements:} Fabrizio Colombo is supported by MUR grant Dipartimento di Eccellenza 2023-2027. Peter Schlosser was funded by the Austrian Science Fund (FWF) under Grant No. J 4685-N and by the European Union--NextGenerationEU.

\section{Introduction}

The $H^\infty$-functional calculus for sectorial operators was introduced by A. McIntosh in \cite{McI1} and has been extensively studied in subsequent works such as \cite{MC10,MC97,MC06,MC98}. This calculus extends the holomorphic functional calculus, see \cite{RD,HP57}, to unbounded operators. It plays a significant role in the spectral theory of differential operators and their applications. Precisely, this calculus is a crucial tool in the theory of maximal regularity for parabolic evolution equations, in resolving Kato's square root problem and in its applications to partial differential equations, including systems of elliptic partial differential operators and Schrödinger operators. For further details see the books \cite{Haase,HYTONBOOK1,HYTONBOOK2}. \medskip

In general, the $H^\infty$-functional calculus defines unbounded operators. However in \cite{CDMY} the authors M. Cowling, J. Doust, A. McIntosh, and A. Yagi employed ideas from harmonic analysis to characterize its boundedness for sectorial operators via suitable quadratic estimates. \medskip

In a Banach module $V$ over the Clifford algebra $\mathbb{R}_n$, the spectral theory for Clifford operators $T:\dom(T)\subseteq V\rightarrow V$ differs from the classical complex theory due to the noncommutativity of the spectral parameter $s$, which is a paravector, and the operator $T$. In this framework, the left and the right $S$-resolvent operators
\begin{equation}\label{Eq_S_resolvent_operators}
S_L^{-1}(s,T):=Q_s[T]^{-1}\overline{s}-TQ_s[T]^{-1}\qquad\text{and}\qquad S_R^{-1}(s,T):=(\overline{s}-T)Q_s[T]^{-1},
\end{equation}
involve the inverse of the second order operator
\begin{equation}\label{Eq_Qs}
Q_s[T]:=T^2-2s_0T+|s|^2,\qquad\text{with }\dom(Q_s[T])=\dom(T^2).
\end{equation}
Hence the corresponding spectrum (called the $S$-spectrum) was introduced in \cite{ColSab2006,ColomboSabadiniStruppa2011}, see also the books \cite{FJBOOK,CGK}, and is based on the invertibility of this exact operator, namely
\begin{equation}\label{Eq_S_spectrum}
\rho_S(T):=\big\{s\in\mathbb{R}^{n+1}\;\big|\;Q_s[T]^{-1}\in\mathcal{B}(V)\big\}\qquad\text{and}\qquad\sigma_S(T):=\mathbb{R}^{d+1}\setminus\rho_S(T),
\end{equation}
where $\mathcal{B}(V)$ denotes the set of bounded, everywhere defined operators. \medskip

Quaternionic and Clifford operators are fundamental in various areas of mathematics and physics. For example, they appear in quaternionic quantum mechanics, see \cite{adler,BF}, as well as in vector analysis, differential geometry, and hypercomplex analysis. We also point out that the spectral theorem based on the $S$-spectrum holds also for quaternionic and Clifford operators, see \cite{ACK,ColKim}. For further details, see the concluding remarks at the end of this paper. \medskip

Due to the noncommutative nature of the Clifford numbers, there are two $S$-resolvent operators \eqref{Eq_S_resolvent_operators} needed to define the $S$-functional calculus, that is the analogue
of the Riesz-Dunford functional calculus, in this setting. Specifically, for left slice hyperholomorphic functions, we utilize the left $S$-resolvent operator, while for right slice hyperholomorphic functions, the right $S$-resolvent operator is employed, to define
\begin{equation}\label{Eq_S_functional_calculus}
f(T):=\frac{1}{2\pi}\int_{\partial U\cap\mathbb{C}_J}S_L^{-1}(s,T)ds_Jf(s)\quad\text{and}\quad f(T):=\frac{1}{2\pi}\int_{\partial U\cap\mathbb{C}_J}f(s)ds_JS_R^{-1}(s,T),
\end{equation}
where $U$ is a suitable open set that contains the $S$-spectrum and $J$ is one of the imaginary units of $\mathbb{R}_n$. \medskip

The generalization of \eqref{Eq_S_functional_calculus} to an $H^\infty$-functional calculus for at least left slice
hyperholomorphic functions was introduced in 2016, see \cite{ACQS2016,CGdiffusion2018,FJBOOK}. It applies to sectorial or bisectorial right linear operators $T$ and to functions $f$ that may grow polynomially at zero and at infinity. The calculus is constructed by the product
\begin{equation}\label{Eq_Hinfty}
f(T):=e(T)^{-1}(ef)(T),
\end{equation}
of the inverse the regularizer $e(T)$, with the regularized operator $(ef)(T)$. Both expressions are understood as the $\omega$-functional calculus, a calculus for bisectorial operator similar to \eqref{Eq_S_functional_calculus}. This calculus is then well-defined since it can be shown that it is independently of the choice of the regularizer $e$. Detailed properties were examined in recent work \cite{MS24} inspired by the universality property of the $S$-functional calculus \cite{ADVCGKS}. \medskip

A persistent challenge has been extending this calculus to right slice hyperholomorphic functions. Due to the noncommutative multiplication of the Clifford numbers, this can not be done in a symmetric way. As noted in \cite[Remark 7.2.2]{FJBOOK}, this definition would depend on the choice of the regularizer function $e$. In the recent paper \cite{RIGHT}, we addressed this unresolved issue by demonstrating that right linear operators can also admit an $H^\infty$-functional calculus for right slice hyperholomorphic functions, utilizing the framework of multivalued operators (also known as right linear relations). \medskip

The aim of this paper is now to investigate the boundedness of the $H^\infty$-functional calculus \eqref{Eq_Hinfty}. In particular, we find certain quadratic estimates which characterize the boundedness of the $H^\infty$-functional calculus in the sense that for every bounded function $f$, also $f(T)$ is bounded. The precise statement of this characterization, see also Theorem~\ref{thm_Quadratic_estimates}, is as follows: Let $T:D(T)\subseteq V\rightarrow V$ be closed, injective and bisectorial of angle $\omega\in(0,\frac{\pi}{2})$. For some $\theta\in(\omega,\frac{\pi}{2})$ let furthermore $\mathcal{N}^0(D_\theta)$ and $N^\infty(D_\theta)$ be the function spaces from Definition~\ref{defi_N0} and Definition~\ref{defi_Ninfty}. Then the following statements are equivalent: \medskip

\begin{enumerate}
\item[i)] There exists a constant $c\geq 0$, such that $f(T)\in\mathcal{B}(V)$ and
\begin{equation*}
\Vert f(T)\Vert\leq c\Vert f\Vert_\infty,\qquad f\in\mathcal{N}^\infty(D_\theta);
\end{equation*}

\item[ii)] For some/each $0 \ne f\in\mathcal{N}^0(D_\theta)$, there exist $c_f,d_f>0$, such that for every $v\in V$
\begin{equation*}
c_f\Vert v\Vert\leq\bigg(\int_{-\infty}^\infty\Vert f(tT)v\Vert^2\frac{dt}{|t|}\bigg)^{\frac{1}{2}}\leq d_f\Vert v\Vert.
\end{equation*}

\item[iii)] For some/each $0\neq f\in\mathcal{N}^0(D_\theta)$, there exist $d_f,\widetilde{d}_f\geq 0$, such that for every $v\in V$
\begin{equation*}
\bigg(\int_{-\infty}^\infty\Vert f(tT)v\Vert^2\frac{dt}{|t|}\bigg)^{\frac{1}{2}}\leq d_f\Vert v\Vert\qquad\text{and}\qquad\bigg(\int_{-\infty}^\infty\Vert f(tT^*)v\Vert^2\frac{dt}{|t|}\bigg)^{\frac{1}{2}}\leq\widetilde{d}_f\Vert v\Vert.
\end{equation*}
\end{enumerate}

\medskip The plan of the paper: Section~\ref{sec_Preliminaries} contains preliminaries on Clifford algebras and Clifford modules. In Section~\ref{sec_Bisectorial_operators}, we survey the basic facts on bisectorial operators and the $S$-functional calculus. In Section~\ref{sec_Quadratic_estimates}, we will prove the main result, the characterization of bounded operators defined via the $H^\infty$-functional calculus in Theorem~\ref{thm_Quadratic_estimates}. Section~\ref{sec_Quadratic_estimates} puts our results into perspective with some concluding remarks on applications and research directions.

\section{Preliminaries on Clifford Algebras and Clifford modules}\label{sec_Preliminaries}

Clifford algebras are associative algebras that generalize complex numbers and quaternions. They provide a mathematical framework for studying operator theory in a noncommutative setting. There are various approaches to extending spectral theory from complex operators to vector operators. The setting we focus on is that of slice hyperholomorphicity, which offers the advantage that $n$-tuples, or more generally $2^n$-tuples, of noncommuting operators can be embedded into a Clifford algebra. With this notion of holomorphicity in the Clifford setting, the concept of spectrum for Clifford operators becomes both very clear and explicit. \medskip

We now fix the algebraic and functional analytic setting in which we work. The underlying algebra we will consider is the real \textit{Clifford algebra} $\mathbb{R}_n$ over $n$ \textit{imaginary units} $e_1,\dots,e_n$, which satisfy the relations
\begin{equation*}
e_i^2=-1\qquad\text{and}\qquad e_ie_j=-e_je_i,\qquad i\neq j\in\{1,\dots,n\}.
\end{equation*}
More precisely, $\mathbb{R}_n$ is given by
\begin{equation*}
\mathbb{R}_n:=\Big\{\sum\nolimits_{A\in\mathcal{A}}s_Ae_A\;\Big|\;s_A\in\mathbb{R},\,A\in\mathcal{A}\Big\},
\end{equation*}
using the index set
\begin{equation*}
\mathcal{A}:=\big\{(i_1,\dots,i_r)\;\big|\;r\in\{0,\dots,n\},\,1\leq i_1<\dots<i_r\leq n\big\},
\end{equation*}
and the \textit{basis vectors} $e_A:=e_{i_1}\dots e_{i_r}$. Note that for $A=\emptyset$ the empty product of imaginary units is the real number $e_\emptyset=1$. Furthermore, we define for every Clifford number $s\in\mathbb{R}_n$ its \textit{conjugate} and its \textit{absolute value}
\begin{equation}\label{Eq_Conjugate_Norm}
\overline{s}:=\sum\nolimits_{A\in\mathcal{A}}(-1)^{\frac{|A|(|A|+1)}{2}}s_Ae_A\qquad\text{and}\qquad|s|^2:=\sum\nolimits_{A\in\mathcal{A}}s_A^2.
\end{equation}
An important subset of the Clifford numbers are the so called \textit{paravectors}
\begin{equation*}
\mathbb{R}^{n+1}:=\big\{s_0+s_1e_1+\dots+s_ne_n\;\big|\;s_0,s_1,\dots,s_n\in\mathbb{R}\big\}.
\end{equation*}
For any paravector $s\in\mathbb{R}^{n+1}$, we define the \textit{imaginary part}
\begin{equation*}
\Im(s):=s_1e_1+\dots+s_ne_n,
\end{equation*}
and the conjugate and the modulus in \eqref{Eq_Conjugate_Norm} reduce to
\begin{equation*}
\overline{s}=s_0-s_1e_1-\dots-s_ne_n\qquad\text{and}\qquad|s|^2=s_0^2+s_1^2+\dots+s_n^2.
\end{equation*}
The sphere of purely imaginary paravectors with modulus $1$, is defined by
\begin{equation*}
\mathbb{S}:=\big\{J\in\mathbb{R}^{n+1}\;\big|\;J_0=0,\,|J|=1\big\}.
\end{equation*}
Any element $J\in\mathbb{S}$ satisfies $J^2=-1$ and hence the corresponding hyperplane
\begin{equation*}
\mathbb{C}_J:=\big\{x+Jy\;\big|\;x,y\in\mathbb{R}\big\}
\end{equation*}
is an isomorphic copy of the complex numbers. Moreover, for every paravector $s\in\mathbb{R}^{n+1}$ we consider the corresponding \textit{$(n-1)$-sphere}
\begin{equation*}
[s]:=\big\{s_0+J|\Im(s)|\;\big|\;J\in\mathbb{S}\big\}.
\end{equation*}
A subset $U\subseteq\mathbb{R}^{n+1}$ is called \textit{axially symmetric}, if $[s]\subseteq U$ for every $s\in U$. \medskip

Next, we introduce the notion of slice hyperholomorphic functions $f:U\rightarrow\mathbb{R}_n$, defined on an axially symmetric open set $U\subseteq\mathbb{R}^{n+1}$.

\begin{defi}[Slice hyperholomorphic functions]
Let $U\subseteq\mathbb{R}^{n+1}$ be open, axially symmetric and consider
\begin{equation*}
\mathcal{U}:=\big\{(x,y)\in\mathbb{R}^2\;\big|\;x+\mathbb{S}y\subseteq U\big\}.
\end{equation*}
A function $f:U\rightarrow\mathbb{R}_n$ is called \textit{left} (resp. \textit{right}) \textit{slice hyperholomorphic}, if there exist continuously differentiable functions $f_0,f_1:\mathcal{U}\rightarrow\mathbb{R}_n$, such that for every $(x,y)\in\mathcal{U}$:

\begin{enumerate}
\item[i)] The function $f$ admits for every $J\in\mathbb{S}$ the representation
\begin{equation*}
f(x+Jy)=f_0(x,y)+Jf_1(x,y),\quad\Big(\text{resp.}\;f(x+Jy)=f_0(x,y)+f_1(x,y)J\Big).
\end{equation*}

\item[ii)] The functions $f_0,f_1$ satisfy the \textit{compatibility conditions}
\begin{equation*}
f_0(x,-y)=f_0(x,y)\qquad\text{and}\qquad f_1(x,-y)=-f_1(x,y).
\end{equation*}

\item[iii)] The functions $f_0,f_1$ satisfy the \textit{Cauchy-Riemann equations}
\begin{equation*}
\frac{\partial}{\partial x}f_0(x,y)=\frac{\partial}{\partial y}f_1(x,y)\qquad\text{and}\qquad\frac{\partial}{\partial y}f_0(x,y)=-\frac{\partial}{\partial x}f_1(x,y).
\end{equation*}
\end{enumerate}

The class of left (resp. right) slice hyperholomorphic functions on $U$ is denoted by $\mathcal{SH}_L(U)$ (resp. $\mathcal{SH}_R(U)$). In the special case that $f_0$ and $f_1$ are real valued, we call $f$ \textit{intrinsic} and the space of intrinsic functions is denoted by $\mathcal{N}(U)$.
\end{defi}

Next, we turn our attention to Clifford modules over $\mathbb{R}_n$. For a real Hilbert space $V_\mathbb{R}$ with inner product $\langle\cdot,\cdot\rangle_\mathbb{R}$ and corresponding norm $\Vert\cdot\Vert_\mathbb{R}^2=\langle\cdot,\cdot\rangle_\mathbb{R}$, we define the corresponding \textit{Clifford module}
\begin{equation*}
V:=\Big\{\sum\nolimits_{A\in\mathcal{A}}v_A\otimes e_A\;\Big|\;v_A\in V_\mathbb{R},\,A\in\mathcal{A}\Big\},
\end{equation*}
and equip it with the \textit{inner product}
\begin{equation}\label{Eq_Inner_product}
\langle v,w\rangle:=\sum\nolimits_{A,B\in\mathcal{A}}\langle v_A,w_B\rangle_\mathbb{R}\,\overline{e_A}\,e_B,\qquad v,w\in V,
\end{equation}
and the generated \textit{norm}
\begin{equation*}
\Vert v\Vert^2:=\Sc\langle v,v\rangle=\sum\nolimits_{A\in\mathcal{A}}\Vert v_A\Vert_\mathbb{R}^2,\qquad v\in V.
\end{equation*}
For any vector $v=\sum_{A\in\mathcal{A}}v_A\otimes e_A\in V$ and any Clifford number $s=\sum_{B\in\mathcal{A}}s_Be_B\in\mathbb{R}_n$, we establish the left and the right scalar multiplication
\begin{align*}
sv:=&\sum\nolimits_{A,B\in\mathcal{A}}(s_Bv_A)\otimes(e_Be_A),\qquad\textit{(left-multiplication)} \\
vs:=&\sum\nolimits_{A,B\in\mathcal{A}}(v_As_B)\otimes(e_Ae_B).\qquad\textit{(right-multiplication)}
\end{align*}
From \cite[Lemma 2.1]{CMS24} we recall the well known properties of these products
\begin{align*}
\Vert sv\Vert&\leq 2^{\frac{n}{2}}|s|\Vert v\Vert\qquad\text{and}\qquad\Vert vs\Vert\leq 2^{\frac{n}{2}}|s|\Vert v\Vert,\qquad\text{if }s\in\mathbb{R}_n, \\
\Vert sv\Vert&=\Vert vs\Vert=|s|\Vert v\Vert,\hspace{4.83cm}\text{if }s\in\mathbb{R}^{n+1}.
\end{align*}
The sesquilinear form \eqref{Eq_Inner_product} is now also right-linear in the second, and right-antilinear in the first argument, i.e. for every $u,v,w\in V$, $s\in\mathbb{R}_n$, there holds
\begin{align*}
\langle u,v+w\rangle&=\langle u,v\rangle+\langle u,w\rangle, && \langle v,ws\rangle=\langle v,w\rangle s, \\
\langle v+w,u\rangle&=\langle v,u\rangle+\langle w,u\rangle, && \langle vs,w\rangle=\overline{s}\langle v,w\rangle.
\end{align*}
Moreover, there also holds
\begin{equation*}
\langle v,sw\rangle=\langle\overline{s}v,w\rangle.
\end{equation*}

\section{Bisectorial operators and the $H^\infty$-functional calculus}\label{sec_Bisectorial_operators}

For any Hilbert module $V$ we will denote the set of \textit{bounded, everywhere defined, right-linear operators}
\begin{equation*}
\mathcal{B}(V):=\big\{T:V\rightarrow V\text{ right-linear}\;\big|\;\dom(T)=V,\,T\text{ is bounded}\big\},
\end{equation*}
as well as the space of \textit{closed operators}
\begin{equation*}
\mathcal{K}(V):=\big\{T:V\rightarrow V\text{ right-linear}\;\big|\;\dom(T)\subseteq V\text{ is right linear},\,T\text{ is closed}\big\}.
\end{equation*}
In order to introduce bisectorial operators we consider for every $\omega\in(0,\frac{\pi}{2})$ the \textit{double sector} \medskip

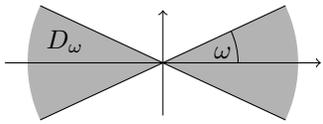
\begin{minipage}{0.34\textwidth}
\begin{center}
\begin{tikzpicture}
\fill[black!30] (0,0)--(1.63,0.76) arc (25:-25:1.8)--(0,0)--(-1.63,-0.76) arc (205:155:1.8);
\draw (-1.63,0.76)--(1.63,-0.76);
\draw (-1.63,-0.76)--(1.63,0.76);
\draw[->] (-2.1,0)--(2.1,0);
\draw[->] (0,-0.7)--(0,0.7);
\draw (1,0) arc (0:25:1) (0.8,-0.07) node[anchor=south] {\small{$\omega$}};
\draw (-1.3,0) node[anchor=south] {\small{$D_\omega$}};
\end{tikzpicture}
\end{center}
\end{minipage}
\begin{minipage}{0.65\textwidth}
\begin{equation}\label{Eq_Domega}
D_\omega:=\big\{re^{J\phi}\;\big|\;r>0,\,J\in\mathbb{S},\,\phi\in I_\omega\big\},
\end{equation}
using the union of intervals
\begin{equation*}
I_\omega:=(-\omega,\omega)\cup(\pi-\omega,\pi+\omega).
\end{equation*}
\end{minipage}

\begin{defi}[Bisectorial operators]
An operator $T\in\mathcal{K}(V)$ is called \textit{bisectorial of angle} $\omega\in(0,\frac{\pi}{2})$, if its $S$-spectrum \eqref{Eq_S_spectrum} is contained in the closed double sector
\begin{equation*}
\sigma_S(T)\subseteq\overline{D_\omega},
\end{equation*}
and for every $\varphi\in(\omega,\frac{\pi}{2})$ there is
\begin{equation}\label{Eq_SL_estimate}
C_\varphi:=\sup\limits_{s\in\mathbb{R}^{n+1}\setminus(D_\varphi\cup\{0\})}|s|\Vert S_L^{-1}(s,T)\Vert<\infty.
\end{equation}
\end{defi}

Next, we define the class of functions for which the following $\omega$-functional calculus is defined.

\begin{defi}\label{defi_N0}
For every $\theta\in(0,\pi)$ let $D_\theta$ be the double sector from \eqref{Eq_Domega}. Then we define the function space
\begin{equation}\label{Eq_N0}
\mathcal{N}^0(D_\theta):=\bigg\{f\in\mathcal{N}(D_\theta)\;\bigg|\;\exists\alpha>0,\,C_\alpha\geq 0: |f(s)|\leq\frac{C_\alpha|s|^\alpha}{1+|s|^{2\alpha}}\bigg\}.
\end{equation}
\end{defi}

Note that this class is a little smaller than the one in \cite[Definition 3.1]{CMS24}. However, in view of the characterization in Theorem~\ref{thm_Quadratic_estimates}, it is enough to consider these classes of functions. In accordance to \cite[Definition 3.5]{CMS24}, we define for those functions the so called \textit{$\omega$-functional calculus}.

\begin{defi}[$\omega$-functional calculus]
Let $T\in\mathcal{K}(V)$ be bisectorial of angle $\omega\in(0,\frac{\pi}{2})$. Then for every $f\in\mathcal{N}^0(D_\theta)$ we define the \textit{$\omega$-functional calculus}
\begin{equation}\label{Eq_omega_functional_calculus}
f(T):=\frac{1}{2\pi}\int_{\partial D_\varphi\cap\mathbb{C}_J}S_L^{-1}(s,T)ds_Jf(s),
\end{equation}
where $\varphi\in(\omega,\theta)$ and $J\in\mathbb{S}$ are arbitrary and the integral \eqref{Eq_omega_functional_calculus} is independent of that choice.
\end{defi}

Note, the the path integral \eqref{Eq_omega_functional_calculus} is understood in the sense
\begin{equation}\label{Eq_omega_parametrized}
f(T)=\frac{1}{2\pi}\sum\limits_\pm\int_{-\infty}^\infty S_L^{-1}(re^{\pm J\varphi},T)\frac{\mp\sgn(r)e^{\pm J\varphi}}{J}f(re^{\pm J\varphi})dr.
\end{equation}
According to \cite[Theorem~3.11]{CMS24}, this $\omega$-functional calculus now satisfies the important product rule
\begin{equation}\label{Eq_Product_rule}
(gf)(T)=g(T)f(T),\qquad g,f\in\mathcal{N}^0(D_\theta),
\end{equation}
which can now be used to enlarge the class of functions \eqref{Eq_N0} for which the functional calculus $f(T)$ is defined. Principally it is possible to allow left and right slice hyperholomorphic functions which are polynomially growing at zero and at infinity, see \cite[Proposition 5.2]{CMS24} and \cite[Definition 4.7]{RIGHT}. However, in this paper we will only concentrate on the following classes of bounded intrinsic functions.

\begin{defi}\label{defi_Ninfty}
For every $\theta\in(0,\frac{\pi}{2})$ let $D_\theta$ be the double sector from \eqref{Eq_Domega}. Then we define the function space of bounded, intrinsic functions
\begin{equation*}
\mathcal{N}^\infty(D_\theta):=\big\{f\in\mathcal{N}(D_\theta)\;\big|\;f\text{ is bounded}\big\}.
\end{equation*}
\end{defi}

In the fashion of \cite[Definition 5.3]{CMS24}, we can then define for those functions the $H^\infty$-functional calculus.

\begin{defi}[$H^\infty$-functional calculus]\label{defi_Hinfty_functional_calculus}
Let $T\in\mathcal{K}(V)$ be injective and bisectorial of angle $\omega\in(0,\frac{\pi}{2})$. Then for every $f\in\mathcal{N}^\infty(D_\theta)$, we define the \textit{$H^\infty$-functional calculus}
\begin{equation}\label{Eq_Hinfty_functional_calculus}
f(T):=e(T)^{-1}(ef)(T),
\end{equation}
where $e(s):=\frac{s}{1+s^2}\in\mathcal{N}^0(D_\theta)$ and the corresponding operator $e(T)$ via \eqref{Eq_omega_functional_calculus} is injective.
\end{defi}

Finally we recall the definition of the adjoint operator in the Clifford setting

\begin{defi}
For every $T\in\mathcal{K}(V)$ with $\overline{\dom}(T)=V$, we define the \textit{adjoint operator} $T^*\in\mathcal{K}(V)$ as the unique operator with the property
\begin{equation*}
\langle T^*w,v\rangle=\langle w,Tv\rangle,\qquad v\in\dom(T),w\in\dom(T^*),
\end{equation*}
and with the domain
\begin{equation*}
\dom(T^*):=\big\{w\in V\;\big|\;\exists v_w\in V: \langle v_w,v\rangle=\langle w,Tv\rangle,\,v\in\dom(T)\big\}.
\end{equation*}
\end{defi}

An important result, which we will need in the sequel, is regarding the $S$-spectrum and the $S$-functional calculus of the adjoint operator. For a proof see \cite{ADJOINT}.

\begin{thm}\label{thm_Adjoint_omega_functional_calculus}
Let $V$ be a Hilbert module and $T\in\mathcal{K}(V)$ be injective and bisectorial of angle $\omega\in(0,\frac{\pi}{2})$. Then also $T^*$ is injective and bisectorial of the same angle $\omega$. Moreover, for every $f\in\mathcal{N}^\infty(D_\theta)$, $\theta\in(\omega,\frac{\pi}{2})$, the $H^\infty$-functional calculus \eqref{Eq_Hinfty_functional_calculus} of $T^*$ connects to the one for $T$ in the sense
\begin{equation*}
f(T^*)=f(T)^*.
\end{equation*}
\end{thm}

\section{Quadratic estimates}\label{sec_Quadratic_estimates}

In the previous section we have now introduced the main definitions regarding the $H^\infty$-functional calculus and the spaces of slice hyperholomorphic functions that we will use. In this section, we will leverage these facts to prove the main results of this paper, which are the characterization of bounded operators defined via the $H^\infty$-functional calculus. \medskip

It is clear from the definition of the $\omega$-functional calculus \eqref{Eq_omega_functional_calculus}, that every $f\in\mathcal{N}^0(D_\theta)$, the operator $f(T)$ is bounded and everywhere defined. This section investigates the question, what happens to the boundedness of the operator $f(T)$, for bounded functions $f\in\mathcal{N}^\infty(D_\theta)$, defined via the unbounded $H^\infty$-functional calculus \eqref{Eq_Hinfty_functional_calculus}. 

\medskip
The central outcome of this section is Theorem~\ref{thm_Quadratic_estimates}, which characterizes the operators $T: \dom(T) \subset V \to V $ for which the functional calculus $f(T)$ is bounded. This theorem is established in the context where $V$ is a Hilbert module. To prove this theorem, we first need to establish several preparatory results.

\begin{lem}\label{lem_Convergence_lemma}
Let $T\in\mathcal{K}(V)$ be bisectorial of angle $\omega\in(0,\frac{\pi}{2})$, and $(f_n)_n\in\mathcal{N}^0(D_\theta)$, for some $\theta\in(\omega,\frac{\pi}{2})$, with
\begin{equation}\label{Eq_Convergence_lemma_boundedness}
M:=\sup\limits_{n\in\mathbb{N}}\Vert f_n\Vert_\infty<\infty\qquad\text{and}\qquad M':=\sup\limits_{n\in\mathbb{N}}\Vert f_n(T)\Vert<\infty.
\end{equation}
Moreover, we assume that there converges
\begin{equation}\label{Eq_Convergence_lemma_convergence}
\lim\limits_{n\rightarrow\infty}f_n=:f,\quad\text{uniformly on compact subsets of }D_\theta.
\end{equation}
Then there is $f\in\mathcal{N}^\infty(D_\theta)$, its $H^\infty$-functional calculus $f(T)\in\mathcal{B}(V)$ is a bounded operator, and there converges
\begin{equation*}
\lim\limits_{n\rightarrow\infty}f_n(T)v=f(T)v,\qquad v\in V.
\end{equation*}
\end{lem}

\begin{proof}
First of all, due to the convergence \eqref{Eq_Convergence_lemma_convergence} and the boundedness \eqref{Eq_Convergence_lemma_boundedness}, there clearly is
\begin{equation*}
|f(s)|=\lim\limits_{n\rightarrow\infty}|f_n(s)|\leq\lim\limits_{n\rightarrow\infty}\Vert f_n\Vert_\infty\leq M,\qquad s\in D_\theta.
\end{equation*}
Since the convergence \eqref{Eq_Convergence_lemma_convergence} is uniform on compact subsets, the limit function is also holomorphic, i.e. we have proven $f\in\mathcal{N}^\infty(D_\theta)$. Let us now consider the regularizer function $e(s)=\frac{s}{1+s^2}$, from Definition~\ref{defi_Hinfty_functional_calculus}. It can now easily be checked that
\begin{equation}\label{Eq_Convergence_lemma_2}
|e(s)|=\frac{|s|}{|1+s^2|}\leq\frac{1}{\cos(\theta)}\frac{|s|}{1+|s|^2},\qquad s\in D_\theta.
\end{equation}
Since $f,(f_n)_n$ are bounded, there is then clearly $ef,(ef_n)_n\in\mathcal{N}^0(D_\theta)$. Consider now some arbitrary $0<a\leq b<\infty$ and $\varphi\in(\omega,\theta)$. Then with the help of the parametrized integral \eqref{Eq_omega_parametrized} and the bound \eqref{Eq_SL_estimate} of the resolvent operator, we can estimate the difference
\begin{align*}
\Vert(ef)&(T)-(ef_n)(T)\Vert \\
&=\frac{1}{2\pi}\bigg\Vert\sum\limits_\pm\int_{-\infty}^\infty S_L^{-1}(re^{\pm J\varphi},T)\frac{\mp\sgn(r)e^{\pm J\varphi}}{J}\big((ef)(re^{\pm J\varphi})-(ef_n)(re^{\pm J\varphi})\big)dr\bigg\Vert \\
&\leq\frac{C_\varphi}{2\pi}\sum_\pm\int_{-\infty}^\infty\big|(ef)(re^{\pm J\varphi})-(ef_n)(re^{\pm J\varphi})\big|\frac{dr}{|r|} \\
&\leq\frac{C_\varphi}{2\pi\cos(\theta)}\sum_\pm\int_{-\infty}^\infty\frac{\big|f(re^{\pm J\varphi})-f_n(re^{\pm J\varphi})\big|}{1+r^2}dr \\
&\leq\frac{2C_\varphi}{\pi\cos(\theta)}\bigg(\int_{(0,a)\cup(b,\infty)}\frac{2M}{1+r^2}dr+\sup\limits_{s\in\overline{D_\varphi},a\leq|s|\leq b}|f(s)-f_n(s)|\int_a^b\frac{1}{1+r^2}dr\bigg).
\end{align*}
On the right hand side we can now use the locally uniform convergence \eqref{Eq_Convergence_lemma_convergence}, to obtain
\begin{equation*}
\limsup\limits_{n\rightarrow\infty}\Vert(ef)(T)-(ef_n)(T)\Vert\leq\frac{4MC_\varphi}{\pi\cos(\theta)}\int_{(0,a)\cup(b,\infty)}\frac{1}{1+r^2}dr.
\end{equation*}
Since moreover $0<a\leq b<\infty$ were arbitrary and the left hand side is independent of $a$ and $b$, we can send $a\rightarrow 0^+$ and $b\rightarrow\infty$, to get
\begin{equation}\label{Eq_Convergence_lemma_3}
\limsup\limits_{n\rightarrow\infty}\Vert(ef)(T)-(ef_n)(T)\Vert\leq\frac{4MC_\varphi}{\pi\cos(\theta)}\lim\limits_{\substack{a\rightarrow 0^+ \\ b\rightarrow\infty}}\int_{(0,a)\cup(b,\infty)}\frac{1}{1+r^2}dr=0.
\end{equation}
By the product rule \eqref{Eq_Product_rule} of the $\omega$-functional calculus, we have
\begin{equation*}
(ef_n)(T)=f_n(T)e(T).
\end{equation*}
However, also for the function $f\in\mathcal{N}^\infty(D_\theta)$, we can use the commutation between $e(T)$ and $(ef)(T)$ from \cite[Corollary~3.18~iii)]{CMS24}, as well as the definition \eqref{Eq_Hinfty_functional_calculus} of the $H^\infty$-functional calculus, to also obtain
\begin{equation*}
(ef)(T)=e(T)^{-1}e(T)(ef)(T)=e(T)^{-1}(ef)(T)e(T)=f(T)e(T).
\end{equation*}
Using these two product rules, the convergence \eqref{Eq_Convergence_lemma_3} becomes
\begin{equation}\label{Eq_Convergence_lemma_4}
\lim\limits_{n\rightarrow\infty}\Vert f(T)w-f_n(T)w\Vert=0,\qquad w\in\ran(e(T)).
\end{equation}
As a direct consequence of this convergence, we obtain from \eqref{Eq_Convergence_lemma_boundedness} the norm estimate
\begin{equation}\label{Eq_Convergence_lemma_5}
\Vert f(T)w\Vert=\lim\limits_{n\rightarrow\infty}\Vert f_n(T)w\Vert\leq\lim\limits_{n\rightarrow\infty}\Vert f_n(T)\Vert\Vert w\Vert\leq M'\Vert w\Vert,\qquad w\in\ran(e(T)).
\end{equation}
Let now $v\in V$. Since
\begin{equation*}
\ran(e(T))=\ran\big(T(1+T^2)^{-1}\big)=\dom(T)\cap\ran(T)
\end{equation*}
is dense in $V$ by \cite[Theorem 3.5]{RIGHT}, there exists a sequence $(v_m)_m\in\ran(e(T))$, which approximates $v=\lim_{m\rightarrow\infty}v_m$. Due to the boundedness \eqref{Eq_Convergence_lemma_5}, also $(f(T)v_m)_m$ is a Cauchy sequence, and hence convergent. Since $f(T)$ is closed, this implies $v\in\dom(f(T))$. Hence we have proven that $\dom(f(T))=V$, and by the closed graph theorem then $f(T)\in\mathcal{B}(V)$. \medskip

Finally, in order to extend the convergence \eqref{Eq_Convergence_lemma_4} to all $v\in V$, we once more use the approximating sequence $(v_m)_m\in\ran(e(T))$ from above. With this sequence, we can estimate
\begin{align*}
\Vert f_n(T)v-f(T)v\Vert&\leq\Vert f_n(T)\Vert\Vert v-v_m\Vert+\Vert f_n(T)v_m-f(T)v_m\Vert+\Vert f(T)\Vert\Vert v_m-v\Vert \\
&\leq(M'+\Vert f(T)\Vert)\Vert v-v_m\Vert+\Vert f_n(T)v_m-f(T)v_m\Vert.
\end{align*}
Since we know that there converges \eqref{Eq_Convergence_lemma_4}, in the limit $n\rightarrow\infty$ this inequality then reduces to
\begin{equation*}
\limsup\limits_{n\rightarrow\infty}\Vert f_n(T)v-f(T)v\Vert\leq(M'+\Vert f(T)\Vert)\Vert v-v_m\Vert.
\end{equation*}
Since $m\in\mathbb{N}$ was chosen arbitrary and the left hand side is independent of $m$, we can also send $m\rightarrow\infty$, and get
\begin{equation*}
\lim\limits_{n\rightarrow\infty}\Vert f_n(T)v-f(T)v\Vert=0,\qquad v\in V. \qedhere
\end{equation*}
\end{proof}

\begin{thm}\label{thm_fab_approximation}
Let $T\in\mathcal{K}(V)$ be bisectorial of angle $\omega\in(0,\frac{\pi}{2})$ and $f\in\mathcal{N}^0(D_\theta)$, $\theta\in(\omega,\frac{\pi}{2})$. Let us define for every $0<a\leq b<\infty$ the function
\begin{equation}\label{Eq_fab}
f_{a,b}(s):=\int_{(-b,-a)\cup(a,b)}f(ts)\frac{dt}{t},\qquad s\in D_\theta.
\end{equation}
Then there is $f_{a,b}\in\mathcal{N}^0(D_\theta)$, and its $\omega$-functional calculus \eqref{Eq_omega_functional_calculus} admits the representation
\begin{equation}\label{Eq_fab_representation}
f_{a,b}(T)=\int_{(-b,-a)\cup(a,b)}f(tT)\frac{dt}{t}.
\end{equation}
Moreover, with the constant $f_{0,\infty}:=\int_{-\infty}^\infty f(t)\frac{dt}{t}$, there converges
\begin{equation}\label{Eq_fab_convergence}
\lim\limits_{\substack{a\rightarrow 0^+ \\ b\rightarrow\infty}}f_{a,b}(T)v=f_{0,\infty}v,\qquad v\in V.
\end{equation}
\end{thm}

\begin{proof}
Since the domain of integration $(-a,-b)\cup(a,b)$ is bounded, it is clear that the integral \eqref{Eq_fab} gives an intrinsic function $f_{a,b}$. In order to show that $f_{a,b}\in\mathcal{N}^0(D_\theta)$, we estimate
\begin{equation}\label{Eq_fab_approximation_1}
|f_{a,b}(s)|\leq\int_{(-b,-a)\cup(a,b)}|f(ts)|\frac{dt}{|t|}\leq 2C_\alpha\int_a^b\frac{|ts|^\alpha}{1+|ts|^{2\alpha}}\frac{dt}{t}.
\end{equation}
Using now the maximum value
\begin{equation}\label{Eq_fab_approximation_3}
\sup\limits_{t>0}\frac{t^\alpha|ts|^\alpha}{(1+t^{2\alpha})(1+|ts|^{2\alpha})}=\frac{|s|^\alpha}{(1+|s|^\alpha)^2}\leq\frac{|s|^\alpha}{1+|s|^{2\alpha}},
\end{equation}
which is attained at $t=|s|^{-\frac{1}{2}}$, we can further estimate \eqref{Eq_fab_approximation_1} by
\begin{equation*}
|f_{a,b}(s)|\leq\frac{2C_\alpha|s|^\alpha}{1+|s|^{2\alpha}}\int_a^b\frac{1+t^{2\alpha}}{t^\alpha}\frac{dt}{t},\qquad s\in D_\theta.
\end{equation*}
Since $0<a\leq b<\infty$, the integral is finite and we have shown that $f_{a,b}\in\mathcal{N}^0(D_\theta)$. The representation \eqref{Eq_fab_representation} then follows directly from the definition of the $\omega$-functional calculus
\begin{align}
f_{a,b}(T)&=\frac{1}{2\pi}\int_{\partial D_\varphi\cap\mathbb{C}_J}S_L^{-1}(s,T)ds_Jf_{a,b}(s) \notag \\
&=\frac{1}{2\pi}\int_{(-b,-a)\cup(a,b)}\int_{\partial D_\varphi\cap\mathbb{C}_J}S_L^{-1}(s,T)ds_Jf(ts)\frac{dt}{t} \notag \\
&=\frac{1}{2\pi}\int_{(-b,-a)\cup(a,b)}\int_{\partial D_\varphi\cap\mathbb{C}_J}S_L^{-1}(s,tT)ds_Jf(s)\frac{dt}{t}=\int_{(-b,-a)\cup(a,b)}f(tT)\frac{dt}{t}, \label{Eq_fab_approximation_4}
\end{align}
where in the third equation we substituted $s\to\frac{s}{t}$ and used the identity
\begin{equation}\label{Eq_fab_approximation_2}
\frac{1}{t}S_L^{-1}\Big(\frac{s}{t},T\Big)=S_L^{-1}(s,tT),\qquad s\in D_\theta,\,t\in\mathbb{R}\setminus\{0\}.
\end{equation}
For the convergence \eqref{Eq_fab_convergence}, we will verify the assumptions of Lemma~\ref{lem_Convergence_lemma} for the functions $f_{a,b}$. To do so, we continue estimating \eqref{Eq_fab_approximation_1} in a different way, namely as
\begin{equation*}
|f_{a,b}(s)|\leq 2C_\alpha\int_0^\infty\frac{|ts|^\alpha}{1+|ts|^{2\alpha}}\frac{dt}{t}=\frac{C_\alpha\pi}{\alpha},\qquad s\in D_\theta.
\end{equation*}
which shows that the first supremum in \eqref{Eq_Convergence_lemma_boundedness} is finite. The more challenging part is the finiteness of the second supremum in \eqref{Eq_Convergence_lemma_boundedness}. To prove this, let us first show that the constant $f_{0,\infty}:=\int_{-\infty}^\infty f(t)\frac{dt}{t}$ can be written as
\begin{equation}\label{Eq_fab_approximation_5}
f_{0,\infty}=\int_{-\infty}^\infty f(ts)\frac{dt}{t},\qquad s\in D_\theta.
\end{equation}
If we write $s=re^{J\varphi}$ for some $r>0$, $J\in\mathbb{S}$, $\varphi\in I_\theta$, we can substitute $t\rightarrow\frac{t}{r}$, and rewrite
\begin{equation*}
\int_{-\infty}^\infty f(ts)\frac{dt}{t}=\int_{-\infty}^\infty f(tre^{J\varphi})\frac{dt}{t}=\int_{-\infty}^\infty f(te^{J\varphi})\frac{dt}{t}=\lim\limits_{\substack{a\rightarrow 0^+ \\ b\rightarrow\infty}}\bigg(\int_{-b}^{-a}+\int_a^b\bigg)f(te^{J\varphi})\frac{dt}{t}.
\end{equation*}
Using now the Cauchy theorem, we can change the integration path according to \medskip

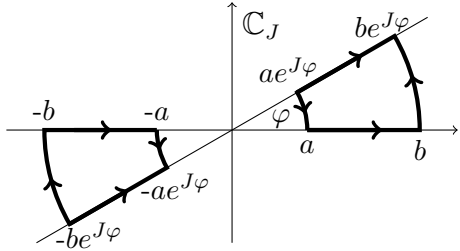
\begin{center}
\begin{tikzpicture}
\draw[->] (-3,0)--(3,0);
\draw[->] (0,-1.5)--(0,1.7) node[anchor=north west] {\large{$\mathbb{C}_J$}};
\draw (-2.6,-1.5)--(2.6,1.5);
\draw[ultra thick] (1,0) arc (0:30:1)--(2.16,1.25) arc (30:0:2.5)--(1,0);
\draw[ultra thick,->] (1,0)--(2,0);
\draw[ultra thick,->] (0.86,0.5)--(1.73,1);
\draw[ultra thick,->] (0.96,0.26)--(0.97,0.22);
\draw[ultra thick,->] (2.41,0.64)--(2.39,0.73);
\draw (1,0) node[anchor=north] {$a$};
\draw (2.5,0) node[anchor=north] {$b$};
\draw (0.75,0.5) node[anchor=south] {$ae^{J\varphi}$};
\draw (2,1.1) node[anchor=south] {$be^{J\varphi}$};
\draw[ultra thick] (-1,0) arc (180:210:1)--(-2.16,-1.25) arc (210:180:2.5)--(-1,0);
\draw[ultra thick,->] (-2.1,0)--(-1.6,0);
\draw[ultra thick,->] (-2.16,-1.25)--(-1.38,-0.8);
\draw[ultra thick,->] (-0.97,-0.22)--(-0.96,-0.26);
\draw[ultra thick,->] (-2.39,-0.73)--(-2.41,-0.64);
\draw (-1,0) node[anchor=south] {-$a$};
\draw (-2.5,0) node[anchor=south] {-$b$};
\draw (-0.75,-0.45) node[anchor=north] {-$ae^{J\varphi}$};
\draw (-1.9,-1.1) node[anchor=north] {-$be^{J\varphi}$};
\draw (0.65,-0.1) node[anchor=south] {$\varphi$};
\end{tikzpicture}
\end{center}
\begin{equation}\label{Eq_fab_approximation_6}
\int_{-\infty}^\infty f(ts)\frac{dt}{t}=\lim\limits_{\substack{a\rightarrow 0^+ \\ b\rightarrow\infty}}\bigg(\int_{-be^{J\varphi}}^{-b}+\int_{-b}^{-a}+\int_{-a}^{-ae^{J\varphi}}+\int_{ae^{J\varphi}}^a+\int_a^b+\int_b^{be^{J\varphi}}\bigg)\frac{f(p)}{p}dp.
\end{equation}
In the limit $b\rightarrow\infty$, the first integral in \eqref{Eq_fab_approximation_6} converges as
\begin{equation*}
\bigg|\int_{-be^{J\varphi}}^{-b}\frac{f(p)}{p}dp\bigg|=\bigg|\int_0^\varphi f(-be^{J\phi})d\phi\bigg|\leq\frac{C_\alpha b^\alpha|\varphi|}{1+b^{2\alpha}}\overset{b\rightarrow\infty}{\longrightarrow}0.
\end{equation*}
Analogously, there also vanish the third, fourth and sixth integral in \eqref{Eq_fab_approximation_6}, namely
\begin{equation*}
\lim\limits_{a\rightarrow 0^+}\int_{-a}^{-ae^{J\varphi}}\frac{f(p)}{p}dp=0,\qquad\lim\limits_{a\rightarrow 0^+}\int_{ae^{J\varphi}}^a\frac{f(p)}{p}dp=0,\qquad\lim\limits_{b\rightarrow\infty}\int_b^{be^{J\varphi}}\frac{f(p)}{p}dp=0.
\end{equation*}
These four vanishing limits, reduce the equation \eqref{Eq_fab_approximation_6} to the stated representation \eqref{Eq_fab_approximation_5}
\begin{equation*}
\int_{-\infty}^\infty f(ts)\frac{dt}{t}=\int_{-\infty}^\infty\frac{f(p)}{p}dp=f_{0,\infty}.
\end{equation*}
With respect to the constant $f_{0,\infty}$, we now split the function $f_{a,b}$ in \eqref{Eq_fab} into the three parts
\begin{equation}\label{Eq_fab_approximation_7}
f_{a,b}(s)=\underbrace{\frac{f_{0,\infty}(as)^2}{1+(as)^2}-\int_{-a}^af(ts)\frac{dt}{t}}_{=:f_{a,b}^{(1)}(s)}+\underbrace{\int_{-b}^bf(ts)\frac{dt}{t}-\frac{f_{0,\infty}(bs)^2}{1+(bs)^2}}_{=:f_{a,b}^{(2)}(s)}+\underbrace{\frac{f_{0,\infty}(b^2-a^2)s^2}{(1+(as)^2)(1+(bs)^2)}}_{=:f_{a,b}^{(3)}(s)}.
\end{equation}
Using the upper bounds \eqref{Eq_Convergence_lemma_2} and a similar estimate as in \eqref{Eq_fab_approximation_1}, we obtain
\begin{align}
|f_{a,b}^{(1)}(s)|&\leq\frac{|f_{0,\infty}||as|^2}{|1+(as)^2|}+2C_\alpha\int_0^a\frac{|ts|^\alpha}{1+|ts|^{2\alpha}}\frac{dt}{t} \notag \\
&\leq\frac{|f_{0,\infty}||as|^2}{\cos(\theta)(1+|as|^2)}+2C_\alpha\int_0^a|ts|^\alpha\frac{dt}{t} \notag \\
&=\frac{|f_{0,\infty}||as|^2}{\cos(\theta)(1+|as|^2)}+\frac{2C_\alpha|as|^\alpha}{\alpha}\leq\Big(\frac{|f_{0,\infty}|}{\cos(\theta)}+\frac{2C_\alpha}{\alpha}\Big)|as|^{\min\{\alpha,2\}},\qquad |s|\leq\frac{1}{a}. \label{Eq_fab_approximation_12}
\end{align}
Moreover, we can use the representation \eqref{Eq_fab_approximation_5} of $f_{0,\infty}$ to rewrite $f_{a,b}^{(1)}(s)$ as
\begin{equation*}
f_{a,b}^{(1)}(s)=\frac{f_{0,\infty}(as)^2}{1+(as)^2}-f_{0,\infty}+\int_{\mathbb{R}\setminus(-a,a)}f(ts)\frac{dt}{t}=-\frac{f_{0,\infty}}{1+(as)^2}+\int_{\mathbb{R}\setminus(-a,a)}f(ts)\frac{dt}{t}.
\end{equation*}
Consequently, we can analogously to \eqref{Eq_fab_approximation_12}, also estimate $f_{a,b}^{(1)}(s)$ by
\begin{align}
|f_{a,b}^{(1)}(s)|&\leq\frac{|f_{0,\infty}|}{|1+(as)^2|}
+2C_\alpha\int_a^\infty\frac{|ts|^\alpha}{1+|ts|^{2\alpha}}\frac{dt}{t} \notag \\
&\leq\frac{|f_{0,\infty}|}{\cos(\theta)(1+|as|^2)}
+2C_\alpha\int_a^\infty\frac{1}{|ts|^\alpha}\frac{dt}{|t|} \notag \\
&=\frac{|f_{0,\infty}|}{\cos(\theta)(1+|as|^2)}+\frac{2C_\alpha}{\alpha|as|^\alpha}\leq\Big(\frac{|f_{0,\infty}|}{\cos(\theta)}+\frac{2C_\alpha}{\alpha}\Big)\frac{1}{|as|^{\min\{\alpha,2\}}},\qquad|s|\geq\frac{1}{a}. \label{Eq_fab_approximation_13}
\end{align}
The two estimates \eqref{Eq_fab_approximation_12} and \eqref{Eq_fab_approximation_13} now on the one hand show that $f_{a,b}^{(1)}\in\mathcal{N}^0(S_\theta)$, and if we choose some arbitrary $\varphi\in(\omega,\theta)$ and $J\in\mathbb{S}$, it can also be used to estimate the operator norm of $f_{a,b}^{(1)}(T)$ by
\begin{align}
\Vert f_{a,b}^{(1)}(T)\Vert&=\frac{1}{2\pi}\bigg\Vert\sum\limits_\pm\int_{-\infty}^\infty S_L^{-1}(re^{\pm J\varphi},T)\frac{\mp\sgn(r)e^{\pm J\varphi}}{J}f_{a,b}^{(1)}(re^{\pm J\varphi})dr\bigg\Vert \notag \\
&\leq\frac{C_\varphi}{2\pi}\sum\limits_\pm\int_{-\infty}^\infty|f_{a,b}^{(1)}(re^{\pm J\varphi})|\frac{dr}{|r|} \notag \\
&\leq\frac{C_\varphi}{\pi}\Big(\frac{|f_{0,\infty}|}{\cos(\theta)}+\frac{2C_\alpha}{\alpha}\Big)\bigg(\int_{-\frac{1}{a}}^{\frac{1}{a}}|ar|^{\min\{\alpha,2\}}\frac{dr}{|r|}+\int_{\mathbb{R}\setminus(-\frac{1}{a},\frac{1}{a})}\frac{1}{|ar|^{\min\{\alpha,2\}}}\frac{dr}{|r|}\bigg) \notag \\
&=\frac{2C_\varphi}{\pi\min\{\alpha,2\}}\Big(\frac{|f_{0,\infty}|}{\cos(\theta)}+\frac{2C_\alpha}{\alpha}\Big). \label{Eq_fab_approximation_8}
\end{align}
Analogously, we show that $f_{a,b}^{(2)}\in\mathcal{N}^0(S_\theta)$, and that the operator norm of $f_{a,b}^{(2)}(T)$ is bound by
\begin{equation}\label{Eq_fab_approximation_9}
\Vert f_{a,b}^{(2)}(T)\Vert\leq\frac{2C_\varphi}{\pi\min\{\alpha,2\}}\Big(\frac{|f_{0,\infty}|}{\cos(\theta)}+\frac{2C_\alpha}{\alpha}\Big).
\end{equation}
Finally, for the function $f_{a,b}^{(3)}$ in \eqref{Eq_fab_approximation_7}, it is clear that $f_{a,b}^{(3)}\in\mathcal{N}^0(S_\theta)$. By the polynomial functional calculus \cite[Theorem 3.16]{CMS24}, we can explicitly write the operator $f_{a,b}^{(3)}(T)$ as
\begin{align*}
f_{a,b}^{(3)}(T)&=f_{0,\infty}(b^2-a^2)T^2(1+a^2T^2)^{-1}(1+b^2T^2)^{-1} \\
&=f_{0,\infty}\big((1+a^2T^2)^{-1}-(1+b^2T^2)^{-1}\big) \\
&=f_{0,\infty}\Big(\frac{1}{a^2}Q_{\frac{J}{a}}[T]^{-1}-\frac{1}{b^2}Q_{\frac{J}{b}}[T]^{-1}\Big),
\end{align*}
where in the last line we used the pseudo resolvent operator \eqref{Eq_Qs} and some arbitrary imaginary unit $J\in\mathbb{S}$. In this representation we can now use the norm bound \cite{RIGHT}, for example with the arbitrary angle $\varphi\in(\omega,\theta)$ from above, to estimate
\begin{equation}\label{Eq_fab_approximation_10}
\Vert f_{a,b}^{(3)}(T)\Vert\leq|f_{0,\infty}|\bigg(\frac{1}{a^2}\frac{2C_\varphi^2}{\frac{1}{a^2}}+\frac{1}{b^2}\frac{2C_\varphi^2}{\frac{1}{b^2}}\bigg)=4|f_{0,\infty}|C_\varphi^2.
\end{equation}
If we now combine the three estimates \eqref{Eq_fab_approximation_8}, \eqref{Eq_fab_approximation_9} and \eqref{Eq_fab_approximation_10}, we end up with the estimate of the full operator
\begin{equation*}
\Vert f_{a,b}(T)\Vert\leq\frac{4C_\varphi}{\pi\min\{\alpha,2\}}\Big(\frac{|f_{0,\infty}|}{\cos(\theta)}+\frac{2C_\alpha}{\alpha}\Big)+4|f_{0,\infty}|C_\varphi^2.
\end{equation*}
Since the right hand side is independent of $0<a\leq b<\infty$, we have proven the finiteness of the second supremum in \eqref{Eq_Convergence_lemma_boundedness}. \medskip

Finally, with the representation \eqref{Eq_fab_approximation_5} of the constant $f_{0,\infty}$, we can estimate
\begin{align*}
|f_{a,b}(s)-f_{0,\infty}|&=\bigg|\int_{-a}^af(ts)\frac{dt}{t}+\int_{\mathbb{R}\setminus(-b,b)}f(ts)\frac{dt}{t}\bigg| \\
&\leq 2C_\alpha\int_0^a\frac{|ts|^\alpha}{1+|ts|^{2\alpha}}\frac{dt}{t}+2C_\alpha\int_b^\infty\frac{|ts|^\alpha}{1+|ts|^{2\alpha}}\frac{dt}{t} \\
&=\frac{2C_\alpha}{\alpha}\Big(\arctan(|as|^\alpha)+\frac{\pi}{2}-\arctan(|bs|^\alpha)\Big),\qquad s\in D_\theta.
\end{align*}
Hence we conclude that there converges
\begin{equation*}
\lim\limits_{\substack{a\rightarrow 0^+ \\ b\rightarrow\infty}}f_{a,b}(s)=f_{0,\infty},\quad\text{uniformly on compact subsets of }D_\theta.
\end{equation*}
Hence, all the assumptions of Lemma~\ref{lem_Convergence_lemma} are satisfied and there follows the convergence
\begin{equation*}
\lim\limits_{\substack{a\rightarrow 0^+ \\ b\rightarrow\infty}}f_{a,b}(T)v=f_{0,\infty}(T)v=f_{0,\infty}v,\qquad v\in V,
\end{equation*}
where in the last equation we used that the $H^\infty$-functional calculus $f_{0,\infty}(T)$ of the constant function is the multiplication with the respective constant $f_{0,\infty}$, see \cite[Theorem 5.9]{CMS24}.
\end{proof}

\begin{lem}\label{lem_fg_estimates}
Let $T\in\mathcal{K}(V)$ be bisectorial of angle $\omega\in(0,\frac{\pi}{2})$ and $f,g\in\mathcal{N}^0(D_\theta)$, $\theta\in(\omega,\frac{\pi}{2})$. Then there holds the following estimates: \medskip

\begin{enumerate}
\item[i)] $\Vert f(tT)g(\tau T)\Vert\leq\frac{C_\theta C_\alpha}{\alpha}\Vert g\Vert_\infty$,\hspace{1.45cm} $t,\tau\in\mathbb{R}\setminus\{0\}$; \medskip

\item[ii)] $\int_{-\infty}^\infty\Vert f(tT)g(\tau T)\Vert\frac{dt}{|t|}\leq\frac{C_\theta C_\alpha C_\beta\pi}{2\alpha\beta}$,\qquad $\tau\in\mathbb{R}\setminus\{0\}$; \medskip

\item[iii)] For every function $\Psi\in L^2(\mathbb{R},\frac{dt}{|t|})$, there is
\begin{equation}\label{Eq_fg_estimate}
\int_{-\infty}^\infty\bigg(\int_{-\infty}^\infty\Vert f(tT)g(\tau T)\Vert|\Psi(t)|\frac{dt}{|t|}\bigg)^2\frac{d\tau}{|\tau|}\leq\Big(\frac{C_\theta C_\alpha C_\beta\pi}{2\alpha\beta}\Big)^2\int_{-\infty}^\infty|\Psi(t)|^2\frac{dt}{|t|}.
\end{equation}
\end{enumerate}
Here $C_\theta\geq 0$ is the constant from \eqref{Eq_SL_estimate}, and $\alpha>0$, $C_\alpha\geq 0$ (resp. $\beta>0$, $C_\beta\geq 0$) are the constants from Definition~\ref{defi_N0} of the function $f$ (resp. $g$).
\end{lem}

\begin{proof}
Before we start with the actual proof of i), ii), and iii), let us first note that for every fixed $t\in\mathbb{R}\setminus\{0\}$, the estimate \eqref{Eq_fab_approximation_3} shows that the function $s\mapsto f(ts)$ is an element in $\mathcal{N}^0(D_\theta)$. The $\omega$-functional calculus \eqref{Eq_omega_functional_calculus} of this function is then given by
\begin{equation}\label{Eq_fg_estimates_3}
f(t\,\cdot\,)(T)=\frac{1}{2\pi}\int_{\partial D_\varphi\cap\mathbb{C}_J}S_L^{-1}(s,T)ds_Jf(ts)=\frac{1}{2\pi}\int_{\partial D_\varphi\cap\mathbb{C}_J}S_L^{-1}(s,tT)ds_Jf(s)=f(tT),
\end{equation}
where in the second equation we substituted $s\mapsto\frac{s}{t}$ and used the identity \eqref{Eq_fab_approximation_2}. \medskip

i)\;\;First, we use the representation \eqref{Eq_fg_estimates_3} and the product rule \eqref{Eq_Product_rule}, to combine the functions
\begin{equation*}
f(tT)g(\tau T)=f(t\,\cdot\,)(T)g(\tau\,\cdot\,)(T)=\big(f(t\,\cdot\,)g(\tau\,\cdot\,)\big)(T).
\end{equation*}
For some arbitrary $\varphi\in(\omega,\theta)$, we can now use the parametrized formula \eqref{Eq_omega_parametrized} of the $\omega$-functional calculus, to estimate
\begin{align}
\Vert f(tT)g(\tau T)\Vert&=\frac{1}{2\pi}\bigg\Vert\sum\limits_\pm\int_{-\infty}^\infty S_L^{-1}(re^{\pm J\varphi},T)\frac{\mp\sgn(r)e^{\pm J\varphi}}{J}f(tre^{\pm J\varphi})g(\tau re^{\pm J\varphi})dr\bigg\Vert \notag \\
&\leq\frac{C_\varphi}{2\pi}\sum\limits_\pm\int_{-\infty}^\infty|f(tre^{\pm J\varphi})||g(\tau re^{\pm J\varphi})|\frac{dr}{|r|} \label{Eq_fg_estimates_1} \\
&\leq\frac{C_\varphi C_\alpha}{2\pi}\Vert g\Vert_\infty\sum\limits_\pm\int_{-\infty}^\infty\frac{|tr|^\alpha}{1+|tr|^{2\alpha}}\frac{dr}{|r|}=\frac{C_\varphi C_\alpha}{\alpha}\Vert g\Vert_\infty. \notag
\end{align}
Since it is clear by \eqref{Eq_SL_estimate} that $\varphi\mapsto C_\varphi$ is monotone decreasing, we can take the limit $\varphi\rightarrow\theta$ of the right hand side, and obtain the stated inequality i). \medskip

ii)\;\;Using the second line in the inequality \eqref{Eq_fg_estimates_1}, and integrate it over $\int_{-\infty}^\infty\frac{dt}{|t|}$, gives
\begin{align*}
\int_{-\infty}^\infty\Vert f(tT)g(\tau T)\Vert\frac{dt}{|t|}&\leq\frac{C_\varphi}{2\pi}\sum\limits_\pm\int_{-\infty}^\infty\int_{-\infty}^\infty|f(tre^{\pm J\varphi})||g(\tau re^{\pm J\varphi})|\frac{dr}{|r|}\frac{dt}{|t|} \\
&\leq\frac{C_\varphi C_\alpha C_\beta}{2\pi}\sum\limits_\pm\int_{-\infty}^\infty\int_{-\infty}^\infty\frac{|tr|^\alpha|\tau r|^\beta}{(1+|tr|^{2\alpha})(1+|\tau r|^{2\beta})}\frac{dr}{|r|}\frac{dt}{|t|} \\
&=\frac{C_\varphi C_\alpha C_\beta}{\alpha}\int_{-\infty}^\infty\frac{|\tau r|^\beta}{1+|\tau r|^{2\beta}}\frac{dr}{|r|}=\frac{C_\varphi C_\alpha C_\beta\pi}{2\alpha\beta}.
\end{align*}
Taking again the limit $\varphi\rightarrow\theta$, gives the estimate ii). \medskip

iii) Using Hölder's inequality together and the already proven inequality ii), gives
\begin{align*}
\int_{-\infty}^\infty\bigg(\int_{-\infty}^\infty&\Vert f(tT)g(\tau T)\Vert|\Psi(t)|\frac{dt}{|t|}\bigg)^2\frac{d\tau}{|\tau|} \\
&\leq\int_{-\infty}^\infty\bigg(\int_{-\infty}^\infty\Vert f(tT)g(\tau T)\Vert\frac{dt}{|t|}\bigg)\bigg(\int_{-\infty}^\infty\Vert f(tT)g(\tau T)\Vert|\Psi(t)|^2\frac{dt}{|t|}\bigg)\frac{d\tau}{|\tau|} \\
&\leq\frac{C_\theta C_\alpha C_\beta\pi}{2\alpha\beta}\int_{-\infty}^\infty\int_{-\infty}^\infty\Vert f(tT)g(\tau T)\Vert\frac{d\tau}{|\tau|}|\Psi(t)|^2\frac{dt}{|t|} \\
&\leq\Big(\frac{C_\theta C_\alpha C_\beta\pi}{2\alpha\beta}\Big)^2\int_0^\infty|\Psi(t)|^2\frac{dt}{|t|}. \qedhere
\end{align*}
\end{proof}

\begin{prop}\label{prop_f_infty_estimate}
Let $T\in\mathcal{K}(V)$ be injective and bisectorial of angle $\omega\in(0,\frac{\pi}{2})$. Then for every $g\in\mathcal{N}^0(D_\theta)$, $\theta\in(\omega,\frac{\pi}{2})$, there exists a constant $C_g\geq 0$, such that
\begin{equation}\label{Eq_f_infty_estimate}
\int_{-\infty}^\infty\Vert g(tT)f(T)v\Vert^2\frac{dt}{|t|}\leq C_g^2\Vert f\Vert_\infty^2\int_{-\infty}^\infty\Vert g(tT)v\Vert^2\frac{dt}{|t|},\qquad f\in\mathcal{N}^0(D_\theta),\,v\in V.
\end{equation}
\end{prop}

\begin{proof}
Since the opposite cases are trivial, we will assume $g\not\equiv 0$ and
\begin{equation}\label{Eq_f_infty_estimate_3}
\int_{-\infty}^\infty\Vert g(tT)v\Vert^2\frac{dt}{|t|}<\infty.
\end{equation}
Let us fix $0<a\leq b<\infty$, and with $e(s)=\frac{s}{1+s^2}$, we consider the function $(eg^2)_{a,b}$ from \eqref{Eq_fab}. Using the representation \eqref{Eq_fab_representation} of $(eg^2)_{a,b}(T)$, we then estimate
\begin{align*}
\int_{-\infty}^\infty\Vert g(tT)(eg^2)_{a,b}(T)f(T)v&\Vert^2\frac{dt}{|t|}=\int_{-\infty}^\infty\bigg\Vert g(tT)\int_{(-b,-a)\cup(a,b)}(eg^2)(\tau T)\frac{d\tau}{\tau}f(T)v\bigg\Vert^2\frac{dt}{|t|} \\
&\leq\int_{-\infty}^\infty\bigg(\int_{(-b,-a)\cup(a,b)}\Vert g(tT)g(\tau T)\Vert\Vert(eg)(\tau T)f(T)v\Vert\frac{d\tau}{|\tau|}\bigg)^2\frac{dt}{|t|} \\
&\leq\frac{C_\theta^2\Vert f\Vert_\infty^2}{\cos^2(\theta)}\int_{-\infty}^\infty\bigg(\int_{(-b,-a)\cup(a,b)}\Vert g(tT)g(\tau T)\Vert\Vert g(\tau T)v\Vert\frac{d\tau}{|\tau|}\bigg)^2\frac{dt}{|t|},
\end{align*}
where in the last line we commuted $g(\tau T)$ and $f(T)$, which is allowed by \cite[Corollary 3.18]{CMS24}, and used Lemma~\ref{lem_fg_estimates}~i) together with \eqref{Eq_Convergence_lemma_2}, in order to estimates $e(\tau T)f(T)$. Using now also the inequality \eqref{Eq_fg_estimate} with the function $\Psi(\tau)=\Vert g(\tau T)v\Vert\mathds{1}_{(-b,-a)\cup(a,b)}(\tau)$, we furthermore get
\begin{equation}\label{Eq_f_infty_estimate_1}
\int_{-\infty}^\infty\Vert g(tT)(eg^2)_{a,b}(T)f(T)v\Vert^2\frac{dt}{|t|}\leq\Big(\frac{C_\theta^2C_\beta^2\pi}{2\cos(\theta)\beta^2}\Big)^2\Vert f\Vert_\infty^2\int_{(-b,-a)\cup(a,b)}\Vert g(\tau T)v\Vert^2\frac{d\tau}{|\tau|}.
\end{equation}
Let us now choose two monotone sequences $a_n\rightarrow 0^+$, $b_n\rightarrow\infty$. Then by \eqref{Eq_fab_convergence}, there converges
\begin{equation}\label{Eq_f_infty_estimate_2}
\lim\limits_{n\rightarrow\infty}g(tT)(eg^2)_{a_n,b_n}(T)f(T)v=g(tT)(eg^2)_{0,\infty}f(T)v,\qquad t\in\mathbb{R}\setminus\{0\}.
\end{equation}
In order to show that this limit converges also in $L^2(\mathbb{R},\frac{dt}{|t|})$, we analogously to \eqref{Eq_f_infty_estimate_1} show that for any $m\geq n$ there holds the estimate
\begin{align*}
\int_{-\infty}^\infty\Vert g(tT)&\big((eg^2)_{a_n,b_n}(T)-(eg^2)_{a_m,b_m}(T)\big)f(T)v\Vert^2\frac{dt}{|t|} \\
&\leq\Big(\frac{C_\theta^2C_\beta^2\pi}{2\cos(\theta)\beta^2}\Big)^2\Vert f\Vert_\infty^2\int_{(-b_m,-b_n)\cup(-a_n,-a_m)\cup(a_m,a_n)\cup(b_n,b_m)}\Vert g(\tau T)v\Vert^2\frac{d\tau}{|\tau|}.
\end{align*}
Since $\int_{-\infty}^\infty\Vert g(\tau T)v\Vert^2\frac{d\tau}{|\tau|}<\infty$ is finite by the assumption \eqref{Eq_f_infty_estimate_3}, this estimate shows that the sequence of mappings
\begin{equation*}
t\mapsto g(tT)(eg^2)_{a_n,b_n}(T)f(T)v,
\end{equation*}
is a Cauchy-sequence in the space $L^2(\mathbb{R},\frac{dt}{|t|})$. Consequently, this sequence converges in $L^2(\mathbb{R},\frac{dt}{|t|})$ to its pointwise limit \eqref{Eq_f_infty_estimate_2}. In particular there converges the $L^2$-norms
\begin{equation}\label{Eq_f_infty_estimate_4}
\lim\limits_{n\rightarrow\infty}\int_{-\infty}^\infty\Vert g(tT)(eg^2)_{a_n,b_n}(T)f(T)v\Vert^2\frac{dt}{|t|}=(eg^2)_{0,\infty}^2\int_{-\infty}^\infty\Vert g(tT)f(T)v\Vert^2\frac{dt}{|t|}.
\end{equation}
Performing now the limit $n\rightarrow\infty$ in \eqref{Eq_f_infty_estimate_1}, and considering that the limit of the left hand side is given by \eqref{Eq_f_infty_estimate_4}, we get
\begin{align*}
(eg^2)_{0,\infty}^2\int_{-\infty}^\infty\Vert g(tT)f(T)v\Vert^2\frac{dt}{|t|}&\leq\Big(\frac{C_\theta^2C_\beta^2\pi}{2\cos(\theta)\beta^2}\Big)^2\Vert f\Vert_\infty^2\lim\limits_{n\rightarrow\infty}\int_{(-b_n,-a_n)\cup(a_n,b_n)}\Vert g(\tau T)v\Vert^2\frac{d\tau}{|\tau|} \\
&=\Big(\frac{C_\theta^2C_\beta^2\pi}{2\cos(\theta)\beta^2}\Big)^2\Vert f\Vert_\infty^2\int_{-\infty}^\infty\Vert g(\tau T)v\Vert^2\frac{d\tau}{|\tau|}.
\end{align*}
Finally, since we assumed $g\not\equiv 0$ and since $g(t)\in\mathbb{R}$ because $g$ is intrinsic, the value
\begin{equation*}
(eg^2)_{0,\infty}=\int_{-\infty}^\infty\frac{g(t)^2}{1+t^2}dt>0,
\end{equation*}
from \eqref{Eq_fab_approximation_5}, is positive, and we have proven \eqref{Eq_f_infty_estimate} with the constant $C_g=\frac{C_\theta^2C_\beta^2\pi}{2\cos(\theta)\beta^2(eg^2)_{0,\infty}}$.
\end{proof}

\begin{lem}\label{lem_Psieps}
Let $\Psi:\mathbb{R}\rightarrow V$ be a continuous function. Then for every $\varepsilon>0$ there exists a Borel measurable function $\Psi_\varepsilon:\mathbb{R}\rightarrow V$, such that for every $t\in\mathbb{R}$, there is \medskip

\begin{enumerate}
\item[i)] $\Vert\Psi_\varepsilon(t)\Vert=\Vert\Psi(t)\Vert$, \medskip

\item[ii)] $\Vert\Psi(t)\Vert^2\leq(1+\varepsilon)\Sc\langle\Psi(t),\Psi_\varepsilon(t)\rangle$.
\end{enumerate}
\end{lem}

\begin{proof}
It is clear that for every $t\in\mathbb{R}$ there is
\begin{equation*}
\Vert\Psi(t)\Vert=\sup_{\Vert v\Vert=1}|\Sc\langle\Psi(t),v\rangle|.
\end{equation*}
Hence there exists some $v_t\in V$, with
\begin{equation}\label{Eq_Psieps_1}
\Vert v_t\Vert=1\qquad\text{and}\qquad\Vert\Psi(t)
\Vert\leq\Big(1+\frac{\varepsilon}{2}\Big)|\Sc\langle\Psi(t),v_t\rangle|.
\end{equation}
Let us consider $I_0:=\big\{t\in\mathbb{R}\;\big|\;\Psi(t)=0\big\}$, and define
\begin{equation}\label{Eq_Psieps_2}
\Psi_\varepsilon(t):=0,\qquad t\in I_0.
\end{equation}
Then it is clear that the assertions i), ii) are satisfied for every $t\in I_0$. \medskip

Next, consider $t\in\mathbb{R}\setminus I_0$. First, we can rewrite the inequality \eqref{Eq_Psieps_1} as
\begin{equation*}
\frac{|\Sc\langle\Psi(t),v_t\rangle|}{\Vert\Psi(t)\Vert}\geq\frac{1}{1+\frac{\varepsilon}{2}}.
\end{equation*}
Since $\Psi$ is continuous, also the mapping $t'\mapsto\frac{|\Sc\langle\Psi(t'),v_t\rangle|}{\Vert\Psi(t')\Vert}$ is continuous on the open set $\mathbb{R}\setminus I_0$. Hence, there exists some $\delta_t>0$, such that
\begin{equation*}
\frac{|\Sc\langle\Psi(t'),v_t\rangle|}{\Vert\Psi(t')\Vert}\geq\frac{1}{1+\varepsilon},\qquad t'\in(t-\delta_t,t+\delta_t).
\end{equation*}
Now, the sets $(t-\delta_t,t+\delta_t)$, $t\in\mathbb{R}\setminus I_0$, form an open cover of $\mathbb{R}\setminus I_0$. Since $\mathbb{R}\setminus I_0$ is a Lindelöf space, there exists a countable subcover
\begin{equation*}
\mathbb{R}\setminus I_0=\bigcup\nolimits_{n\in\mathbb{N}}(t_n-\delta_{t_n},t_n+\delta_{t_n}).
\end{equation*}
Setting now
\begin{equation*}
I_1:=(t_1-\delta_{t_1},t_1+\delta_{t_1})\qquad\text{and}\qquad I_n:=(t_n-\delta_{t_n},t_n+\delta_{t_n})\setminus\bigcup\nolimits_{k=1}^{n-1}I_k,\qquad n\geq 2,
\end{equation*}
we obtain a family of disjoint Borel sets, with the property
\begin{equation*}
I=\biguplus\nolimits_{n=1}^\infty I_n\qquad\text{and}\qquad\Vert\Psi(t)\Vert\leq(1+\varepsilon)|\Sc\langle\Psi(t),v_{t_n}\rangle|,\qquad t\in I_n.
\end{equation*}
If we now define
\begin{equation*}
\Phi_\varepsilon(t):=\Vert\Psi(t)\Vert v_{t_n},\qquad t\in I_n,
\end{equation*}
this function satisfies
\begin{equation*}
\Vert\Phi_\varepsilon(t)\Vert=\Vert\Psi(t)\Vert\qquad\text{and}\qquad\Vert\Psi(t)\Vert^2\leq(1+\varepsilon)|\Sc\langle\Psi(t),\Phi_\varepsilon(t)\rangle|,\qquad t\in I_n.
\end{equation*}
If we now choose
\begin{equation}\label{Eq_Psieps_4}
\Psi_\varepsilon(t):=\sgn\big(\Sc\langle\Psi(t),\Phi_\varepsilon(t)\rangle\big)\Phi_\varepsilon(t),\qquad t\in I_n,
\end{equation}
this function $\Psi_\varepsilon(t)$ satisfies both stated properties i), ii). \medskip

Altogether, since $\mathbb{R}=I_0\uplus\biguplus_{n=1}^\infty I_n$, we have defined in \eqref{Eq_Psieps_2} and \eqref{Eq_Psieps_4} the function $\Psi_\varepsilon$ for every $t\in(0,\infty)$ and the lemma is proven.
\end{proof}

Consider now the group $\mathbb{Z}_2=\{-1,1\}$ equipped with the classical multiplication of real numbers and the discrete topology, i.e. the topology where every subset of $\{-1,1\}$ is open. Then the $\mathbb{Z}$-fold direct product
\begin{equation}\label{Eq_G}
G:=\{-1,1\}^\mathbb{Z}=\big\{(a_k)_{k\in\mathbb{Z}}\;\big|\;a_k\in\{-1,1\}\big\}
\end{equation}
is a compact topological group by Tychonoff's theorem. Moreover, we denote by $\mu$ the normalised Haar measure on $G$, i.e. the unique measure measure $\mu$, defined on the Borel $\sigma$-algebra of $G$, with the properties
\begin{equation}\label{Eq_Haar_measure}
\mu(G)=1\qquad\text{and}\qquad\mu(aB)=\mu(B),\qquad a\in G,\,B\subseteq G\text{ Borel set}.
\end{equation}
The upcoming Lemma~\ref{Haar_measure_orthonormality} now proves that the projectors $p_n:G\rightarrow\{-1,1\}$, defined by
\begin{equation}\label{Eq_Projector}
p_n(a):=a_n,\qquad a=(a_k)_k\in G,
\end{equation}
form an orthonormal system with respect to the Haar measure $\mu$. For the sake of completeness, we will also give a proof of this statement.

\begin{lem}\label{Haar_measure_orthonormality}
Let $G$ be as in \eqref{Eq_G} and $\mu$ the Haar measure from \eqref{Eq_Haar_measure}. Then the projectors \eqref{Eq_Projector} satisfy the orthonormality condition
\begin{equation}\label{Eq_Projector_orthonormality}
\int_Gp_n(a)p_m(a)d\mu(a)=\delta_{n,m},\qquad n,m\in\mathbb{Z}.
\end{equation}
\end{lem}

\begin{proof}
Let us define for every $n\in\mathbb{N}$ the sets
\begin{equation*}
A_n:=\big\{a\in G\;\big|\;a_n=1\big\}\qquad\text{and}\qquad B_n:=\big\{a\in G\;\big|\;a_n=-1\big\}.
\end{equation*}
With the characteristic functions $\mathds{1}_{A_n}$ and $\mathds{1}_{B_n}$ of these sets, we are able to write the projector $p_n$ as the difference
\begin{equation*}
p_n=\mathds{1}_{A_n}-\mathds{1}_{B_n},
\end{equation*}
and consequently the product of two projectors as
\begin{equation*}
p_np_m=(\mathds{1}_{A_n}-\mathds{1}_{B_n})(\mathds{1}_{A_m}-\mathds{1}_{B_m})=\mathds{1}_{A_n\cap A_m}-\mathds{1}_{A_n\cap B_m}-\mathds{1}_{B_n\cap A_m}+\mathds{1}_{B_n\cap B_m}.
\end{equation*}
This means that the integral \eqref{Eq_Projector_orthonormality} writes as
\begin{align}
\int_Gp_np_md\mu&=\mu(A_n\cap A_m)-\mu(A_n\cap B_m)-\mu(B_n\cap A_m)+\mu(B_n\cap B_m) \notag \\
&=\mu\big((A_n\cap A_m)\uplus(B_n\cap B_m)\big)-\mu\big((A_n\cap B_m)\uplus(B_n\cap A_m)\big) \notag \\
&=\mu\big\{a\in G\;\big|\;a_n=a_m\big\}-\mu\big\{a\in G\;\big|\;a_n\neq a_m\big\}. \label{Eq_Projector_orthonormality_1}
\end{align}
In the case $n=m$, the right hand side of \eqref{Eq_Projector_orthonormality_1} reduces to
\begin{equation*}
\int_Gp_np_md\mu=\mu\big\{a\in G\;\big|\;a_n=a_n\big\}-\mu\big\{a\in G\;\big|\;a_n\neq a_n\big\}=\mu(G)-\mu(\emptyset)=1.
\end{equation*}
In the case $n\neq m$, we choose the special element $e=(e_k)_k\in G$ with
\begin{equation*}
e_k:=\begin{cases} -1, & k=m, \\ 1, & k\neq m, \end{cases},\qquad k\in\mathbb{Z},
\end{equation*}
and rewrite the first set on the right hand side of \eqref{Eq_Projector_orthonormality_1} as
\begin{align*}
\big\{a\in G\;\big|\;a_n=a_m\big\}&=e\big\{e^{-1}a\;\big|\;a\in G,\,a_n=a_m\big\}=e\big\{b\in G\;\big|\;(eb)_n=(eb)_m\big\} \\
&=e\big\{b\in G\;\big|\;b_n=-b_m\big\}=e\big\{b\in G\;\big|\;b_n\neq b_m\big\}.
\end{align*}
With this representation of the set and the property \eqref{Eq_Haar_measure} of the Haar measure, the right hand side of \eqref{Eq_Projector_orthonormality_1} vanishes, as
\begin{align*}
\int_Gp_np_md\mu&=\mu\big(e\big\{b\in G\;\big|\;b_n\neq b_m\big\}\big)-\mu\big\{a\in G\;\big|\;a_n\neq a_m\big\} \\
&=\mu\big\{b\in G\;\big|\;b_n\neq b_m\big\}-\mu\big\{a\in G\;\big|\;a_n\neq a_m\big\}=0. \qedhere
\end{align*}
\end{proof}

Based on the preparatory results above, we are now in a position to prove our main result, which is the characterization of boundedness of the Clifford $H^\infty$- functional calculus.

\begin{thm}\label{thm_Quadratic_estimates}
Let $V$ be a Hilbert module and $T\in\mathcal{K}(V)$ be injective and bisectorial of angle $\omega\in(0,\frac{\pi}{2})$. Then for every $\theta\in(\omega,\frac{\pi}{2})$, the following statements are equivalent: \medskip

\begin{enumerate}
\item[i)] There exists some $c\geq 0$, such that for every $f\in\mathcal{N}^\infty(D_\theta)$ there is $f(T)\in\mathcal{B}(V)$ and
\begin{equation}\label{Eq_Bounded_Hinfty}
\Vert f(T)\Vert\leq c\Vert f\Vert_\infty;
\end{equation}

\item[ii)] For some/each $0\neq g\in\mathcal{N}^0(D_\theta)$, there exist constants $c_g,d_g>0$, such that
\begin{equation}\label{Eq_Quadratic_estimates}
c_g\Vert v\Vert\leq\bigg(\int_{-\infty}^\infty\Vert g(tT)v\Vert^2\frac{dt}{|t|}\bigg)^{\frac{1}{2}}\leq d_g\Vert v\Vert,\qquad v\in V.
\end{equation}

\item[iii)] For some/each $0\neq g\in\mathcal{N}^0(D_\theta)$, there exist constants $\widetilde{c}_g,\widetilde{d}_g\geq 0$, such that
\begin{subequations}\label{Eq_Quadratic_estimates_T_Tstar}
\begin{align}
\bigg(\int_{-\infty}^\infty\Vert g(tT)v\Vert^2\frac{dt}{|t|}\bigg)^{\frac{1}{2}}&\leq\widetilde{c}_g\Vert v\Vert,\qquad v\in V, \label{Eq_Quadratic_estimates_T} \\
\bigg(\int_{-\infty}^\infty\Vert g(tT^*)v\Vert^2\frac{dt}{|t|}\bigg)^{\frac{1}{2}}&\leq\widetilde{d}_g\Vert v\Vert,\qquad v\in V. \label{Eq_Quadratic_estimates_Tstar}
\end{align}
\end{subequations}
\end{enumerate}
\end{thm}

\begin{proof}
For the implication $i)\Rightarrow iii)$ let $0\neq g\in\mathcal{N}^0(D_\theta)$ be arbitrary, and $v\in V$. Then for every $n\in\mathbb{N}$, there is
\begin{align}
\int_{(-2^n,-2^{-n})\cup(2^{-n},2^n)}\Vert g(tT)v\Vert^2\frac{dt}{|t|}&=\sum_{k=-n}^{n-1}\int_{(-2^{k+1},-2^k)\cup(2^k,-2^{k+1})}\Vert g(tT)v\Vert^2\frac{dt}{|t|} \notag \\
&=\sum_{k=-n}^{n-1}\int_{(-2,-1)\cup(1,2)}\Vert g(t2^kT)v\Vert^2\frac{dt}{|t|}. \label{Eq_Quadratic_estimates_2}
\end{align}
Using the group $G$ in \eqref{Eq_G}, the Haar measure \eqref{Eq_Haar_measure}, and the projectors \eqref{Eq_Projector_orthonormality}, we can use the orthonormality \eqref{Eq_Projector_orthonormality}, to write the sum on the right hand side of \eqref{Eq_Quadratic_estimates_2} as
\begin{align*}
\sum\limits_{k=-n}^{n-1}\Vert g(t2^kT)v\Vert^2&=\sum\limits_{k,l=-n}^{n-1}\int_Gp_k(a)p_l(a)d\mu(a)\Sc\big\langle g(t2^kT)v,g(t2^lT)v\big\rangle \\
&=\int_G\Sc\bigg\langle\sum\limits_{k=-n}^{n-1}p_k(a)g(t2^kT)v,\sum\limits_{l=-n}^{n-1}p_l(a)g(t2^lT)v\bigg\rangle d\mu(a) \\
&=\int_G\bigg\Vert\sum\limits_{k=-n}^{n-1}p_k(a)g(t2^kT)v\bigg\Vert^2d\mu(a) \\
&=\int_G\bigg\Vert\bigg(\sum\limits_{k=-n}^{n-1}p_k(a)g(t2^k\,\cdot\,)\bigg)(T)v\bigg\Vert^2d\mu(a),
\end{align*}
where in the last equation we used the identity $g(t2^kT)=g(t2^k\,\cdot\,)(T)$ from \eqref{Eq_fg_estimates_3} as well as the linearity of the $\omega$-functional calculus. In this form we can now use the assumed boundedness of the $H^\infty$-functional calculus \eqref{Eq_Bounded_Hinfty} with the function $f(s)=\sum_{k=-n}^{n-1}p_k(a)g(t2^ks)$, to estimate
\begin{equation}\label{Eq_Quadratic_estimates_1}
\sum\limits_{k=-n}^{n-1}\Vert g(t2^kT)v\Vert^2\leq c^2\int_G\bigg\Vert\sum\limits_{k=-n}^{n-1}p_k(a)g(t2^k\,\cdot\,)\bigg\Vert_\infty^2d\mu(a)\Vert v\Vert^2.
\end{equation}
For every $a\in G$, $t\in\mathbb{R}\setminus\{0\}$ and $s\in D_\theta$, we can now estimate the $\Vert\cdot\Vert_\infty$-norm in \eqref{Eq_Quadratic_estimates_1} by
\begin{equation*}
\bigg|\sum\limits_{k=-n}^{n-1}p_k(a)g(t2^ks)\bigg|\leq\sum\limits_{k=-n}^{n-1}|g(t2^ks)|\leq C_\beta\sum\limits_{k=-\infty}^\infty\frac{|t2^ks|^\beta}{1+|t2^ks|^{2\beta}}.
\end{equation*}
Choosing now $k_0\in\mathbb{Z}$ such that $|ts|\in[2^{k_0},2^{k_0+1}]$, we can split up the sum to further estimate
\begin{align*}
\bigg|\sum\limits_{k=-n}^{n-1}p_k(a)g(t2^ks)\bigg|&\leq C_\beta\sum\limits_{k=-\infty}^{-k_0-1}|t2^ks|^\beta+C_\beta\sum\limits_{k=-k_0}^\infty\frac{1}{|t2^ks|^\beta} \\
&\leq C_\beta\sum\limits_{k=-\infty}^{-k_0-1}(2^{k+k_0+1})^\beta+C_\beta\sum\limits_{k=-k_0}^\infty\frac{1}{(2^{k+k_0})^\beta}=\frac{2C_\beta}{1-2^{-\beta}}.
\end{align*}
Since this estimate is true for every $s\in D_\theta$, and since $\mu(G)=1$ by \eqref{Eq_Haar_measure}, it turns \eqref{Eq_Quadratic_estimates_1} into
\begin{equation*}
\sum\limits_{k=-n}^{n-1}\Vert g(t2^kT)v\Vert^2\leq c^2\int_G\frac{4C_\beta^2}{(1-2^{-\beta})^2}d\mu(a)\Vert v\Vert^2=\frac{4c^2C_\beta^2}{(1-2^{-\beta})^2}\Vert v\Vert^2.
\end{equation*}
Plugging this now into \eqref{Eq_Quadratic_estimates_2}, gives
\begin{equation*}
\int_{(-2^n,-2^{-n})\cup(2^{-n},2^n)}\Vert g(tT)v\Vert^2\frac{dt}{|t|}\leq\int_{(-2,-1)\cup(1,2)}\frac{4c^2C_\beta^2}{(1-2^{-\beta})^2}\Vert v\Vert^2\frac{dt}{|t|}=\frac{8\ln(2)c^2C_\beta^2}{(1-2^{-\beta})^2}\Vert v\Vert^2.
\end{equation*}
Since this is true for every $n\in\mathbb{N}$, we can take the limit $n\rightarrow\infty$ and obtain the first inequality \eqref{Eq_Quadratic_estimates_T}, namely
\begin{equation}\label{Eq_Quadratic_estimates_3}
\int_{-\infty}^\infty\Vert g(tT)v\Vert^2\frac{dt}{|t|}\leq\frac{8\ln(2)c^2C_\beta^2}{(1-2^{-\beta})^2}\Vert v\Vert^2.
\end{equation}
For the second inequality \eqref{Eq_Quadratic_estimates_T_Tstar}, we note that by Theorem~\ref{thm_Adjoint_omega_functional_calculus} and \eqref{Eq_Bounded_Hinfty}, there is
\begin{equation*}
\Vert g(T^*)\Vert=\Vert g(T)^*\Vert=\Vert g(T)\Vert\leq c\Vert g\Vert_\infty.
\end{equation*}
This means that the assumption \eqref{Eq_Bounded_Hinfty} is also satisfied for the adjoint operator $T^*$. In the same way as for $T$ in \eqref{Eq_Quadratic_estimates_3}, we can now also derive the estimate
\begin{equation*}
\int_{-\infty}^\infty\Vert g(tT^*)v\Vert^2\frac{dt}{|t|}\leq\frac{8\ln(2)c^2C_\beta}{(1-2^{-\beta})^2}\Vert v\Vert^2.
\end{equation*}
Next, let us prove the implication $\text{\grqq}ii)\Rightarrow i)\text{\grqq}$, and let $0\neq g\in\mathcal{N}^0(D_\theta)$ be the one function for which \eqref{Eq_Quadratic_estimates} is satisfied. In a \textit{first step} we will verify \eqref{Eq_Bounded_Hinfty} for $f\in\mathcal{N}^0(D_\theta)$. We know from Proposition~\ref{prop_f_infty_estimate}, that there exists some constant $C_g\geq 0$, such that
\begin{equation*}
\int_{-\infty}^\infty\Vert g(tT)f(T)v\Vert^2\frac{dt}{|t|}\leq C_g^2\Vert f\Vert_\infty^2\int_{-\infty}^\infty\Vert g(tT)v\Vert^2\frac{dt}{|t|}.
\end{equation*}
Thus, with the given inequality \eqref{Eq_Quadratic_estimates} for the vector $f(T)v$, we then get
\begin{equation*}
c_g^2\Vert f(T)v\Vert^2\leq\int_{-\infty}^\infty\Vert g(tT)f(T)v\Vert^2\frac{dt}{|t|}\leq C_g^2\Vert f\Vert_\infty^2\int_{-\infty}^\infty\Vert g(tT)v\Vert^2\frac{dt}{|t|}\leq C_g^2d_g^2\Vert f\Vert_\infty^2\Vert v\Vert^2.
\end{equation*}
Since this estimate is true for every $v\in V$, we conclude
\begin{equation}\label{Eq_Quadratic_estimates_4}
\Vert f(T)\Vert\leq\frac{C_gd_g}{c_g}\Vert f\Vert_\infty,\qquad f\in\mathcal{N}^0(D_\theta).
\end{equation}
In the \textit{second step} we will extend \eqref{Eq_Quadratic_estimates_4} to all $f\in\mathcal{N}^\infty(D_\theta)$. To do so, let us consider for every $0<\alpha\leq 1$ the function
\begin{equation*}
e_\alpha(s):=\begin{cases} \frac{s^\alpha}{(1+s^2)^\alpha}, & \text{if }\Sc(s)>0, \\ \frac{(-s)^\alpha}{(1+s^2)^\alpha}, & \text{if }\Sc(s)<0, \end{cases} \qquad s\in D_\theta.
\end{equation*}
Then there clearly is $e_\alpha\in\mathcal{N}^0(D_\theta)$, and since $f$ is bounded, also $e_\alpha f\in\mathcal{N}^0(D_\theta)$. This means, that from the already proven inequality \eqref{Eq_Quadratic_estimates_4}, we get
\begin{equation}\label{Eq_Quadratic_estimates_5}
\Vert(e_\alpha f)(T)\Vert\leq\frac{C_gd_g}{c_g}\Vert e_\alpha f\Vert_\infty\leq\frac{C_gd_g}{c_g\cos(\theta)}\Vert f\Vert_\infty,
\end{equation}
where in the second inequality, we used \eqref{Eq_Convergence_lemma_2} to obtain the uniform upper bound
\begin{equation*}
|e_\alpha(s)f(s)|\leq\frac{\Vert f\Vert_\infty|s|^\alpha}{|1+s^2|^{\alpha}}\leq\frac{\Vert f\Vert_\infty|s|^\alpha}{\cos^\alpha(\theta)(1+|s|^2)^\alpha}\leq\frac{\Vert f\Vert_\infty}{2^\alpha\cos^\alpha(\theta)}\leq\frac{\Vert f\Vert_\infty}{\cos(\theta)},\qquad s\in D_\theta.
\end{equation*}
Next, it follows from a convergence result due to Vitali \cite[Proposition 5.1.1]{Haase}, that there converges $\lim_{\alpha\rightarrow 0^+}e_\alpha(s)=1$, uniformly on compact subsets of $D_\theta$. This means that from Lemma~\ref{lem_Convergence_lemma}, we then conclude $f(T)\in\mathcal{B}(V)$, as well as the strong convergence
\begin{equation*}
\lim\limits_{\alpha\rightarrow 0^+}(e_\alpha f)(T)v=f(T)v,\qquad v\in V.
\end{equation*}
In particular, using this convergence together with \eqref{Eq_Quadratic_estimates_5}, also $f(T)$ satisfies the norm estimate
\begin{equation*}
\Vert f(T)v\Vert=\lim\limits_{\alpha\rightarrow 0^+}\Vert(e_\alpha f)(T)v\Vert\leq\frac{C_gd_g}{c_g\cos(\theta)}\Vert f\Vert_\infty.
\end{equation*}
For the implication $\text{\grqq}ii)\Rightarrow iii)\text{\grqq}$ let us fix $0\neq g\in\mathcal{N}^0(D_\theta)$ and consider some arbitrary $v\in V$. Then we know from Lemma~\ref{lem_Psieps} with $\Psi(t)=g(tT^*)v$, that for every $\varepsilon>0$ there exists some $\Psi_\varepsilon:\mathbb{R}\rightarrow V$, with
\begin{equation*}
\Vert\Psi_\varepsilon(t)\Vert=\Vert g(tT^*)v\Vert\qquad\text{and}\qquad\Vert g(tT^*)v\Vert^2\leq(1+\varepsilon)\Sc\big\langle g(tT^*)v,\Psi_\varepsilon(t)\big\rangle,\qquad t\in\mathbb{R}.
\end{equation*}
From the assumed lower bound in \eqref{Eq_Quadratic_estimates}, there then follows the estimate
\begin{align}
\int_{-\infty}^\infty\Vert g(tT^*)v\Vert^2\frac{dt}{|t|}&\leq(1+\varepsilon)\int_{-\infty}^\infty\Sc\big\langle g(tT^*)v,\Psi_\varepsilon(t)\big\rangle\frac{dt}{|t|}=(1+\varepsilon)\int_{-\infty}^\infty\Sc\big\langle v,g(tT)\Psi_\varepsilon(t)\big\rangle\frac{dt}{|t|} \notag \\
&=(1+\varepsilon)\Sc\bigg\langle v,\underbrace{\int_{-\infty}^\infty g(tT)\Psi_\varepsilon(t)\frac{dt}{|t|}}_{=:w}\bigg\rangle\leq(1+\varepsilon)\Vert v\Vert\Vert w\Vert. \label{Eq_Quadratic_estimates_7}
\end{align}
Moreover, for this vector $w$, we can now apply the inequality \eqref{Eq_fg_estimate}, to get an estimate between the two quadratic integrals
\begin{align}
\int_{-\infty}^\infty\Vert g(\tau T)w\Vert^2\frac{d\tau}{|\tau|}&=\int_{-\infty}^\infty\bigg\Vert g(\tau T)\int_{-\infty}^\infty g(tT)\Psi_\varepsilon(t)\frac{dt}{|t|}\bigg\Vert^2\frac{d\tau}{|\tau|} \notag \\
&\leq\Big(\frac{C_\theta C_\beta^2\pi}{2\beta^2}\Big)^2\int_{-\infty}^\infty\Vert\Psi_\varepsilon(t)\Vert^2\frac{dt}{|t|} \notag \\
&=\Big(\frac{C_\theta C_\beta^2\pi}{2\beta^2}\Big)^2\int_{-\infty}^\infty\Vert g(tT^*)v\Vert^2\frac{dt}{|t|}. \label{Eq_Quadratic_estimates_8}
\end{align}
Applying now first the estimate \eqref{Eq_Quadratic_estimates_8} and then \eqref{Eq_Quadratic_estimates_7}, gives
\begin{align*}
\bigg(\int_{-\infty}^\infty\Vert g(\tau T)w\Vert^2\frac{d\tau}{|\tau|}\bigg)^{\frac{1}{2}}\bigg(\int_{-\infty}^\infty\Vert g(tT^*)v\Vert^2\frac{dt}{|t|}\bigg)^{\frac{1}{2}}&\leq\frac{C_\theta C_\beta^2\pi}{2\beta^2}\int_{-\infty}^\infty\Vert g(tT^*)v\Vert^2\frac{dt}{|t|} \\
&\leq\frac{(1+\varepsilon)C_\theta C_\beta^2\pi}{2\beta^2}\Vert v\Vert\Vert w\Vert.
\end{align*}
Using now also the lower bound of the quadratic estimate for $w$ in \eqref{Eq_Quadratic_estimates_T}, then gives
\begin{equation*}
\bigg(\int_{-\infty}^\infty\Vert g(tT^*)v\Vert^2\frac{dt}{|t|}\bigg)^{\frac{1}{2}}\leq\frac{(1+\varepsilon)C_\theta C_\beta^2\pi}{2\beta^2c_g}\Vert v\Vert.
\end{equation*}
Finally, for the implication $\text{\grqq}iii)\Rightarrow ii)\text{\grqq}$ let us fix $0\neq g\in\mathcal{N}^0(D_\theta)$ and consider some arbitrary $v\in V$. It then follows from Theorem~\ref{thm_fab_approximation}, as well as the representation \eqref{Eq_fab_approximation_4} of $g_{a,b}^2(T)$, that
\begin{equation*}
(g^2)_{0,\infty}v=\lim\limits_{\substack{a\rightarrow 0^+ \\ b\rightarrow\infty}}g_{a,b}^2(T)v=\lim\limits_{\substack{a\rightarrow 0^+ \\ b\rightarrow\infty}}\int_{(-b,-a)\cup(a,b)}g^2(tT)v\frac{dt}{t}.
\end{equation*}
Taking now the scalar part of the inner product of this equation with the vector $v$, then gives
\begin{align*}
(g^2)_{0,\infty}\Vert v\Vert^2&=\lim\limits_{\substack{a\rightarrow 0^+ \\ b\rightarrow\infty}}\int_{(-b,-a)\cup(a,b)}\Sc\langle g^2(tT)v,v\rangle\frac{dt}{t} \\
&=\lim\limits_{\substack{a\rightarrow 0^+ \\ b\rightarrow\infty}}\int_{(-b,-a)\cup(a,b)}\Sc\langle g(tT)v,g(tT^*)v\rangle\frac{dt}{t} \\
&\leq\int_{-\infty}^\infty\Vert g(tT)v\Vert\Vert g(tT^*)v\Vert\frac{dt}{|t|} \\
&\leq\bigg(\int_{-\infty}^\infty\Vert g(tT)v\Vert^2\frac{dt}{|t|}\bigg)^{\frac{1}{2}}\bigg(\int_{-\infty}^\infty\Vert g(tT^*)v\Vert^2\frac{dt}{|t|}\bigg)^{\frac{1}{2}},
\end{align*}
where in the last line we used Hölder's inequality. In this form we can now use the assumed quadratic estimate \eqref{Eq_Quadratic_estimates_Tstar}, to get the stated lower bound in \eqref{Eq_Quadratic_estimates}
\begin{equation*}
\frac{(g^2)_{0,\infty}}{\widetilde{d}_g}\Vert v\Vert\leq\bigg(\int_{-\infty}^\infty\Vert g(tT)v\Vert^2\frac{dt}{|t|}\bigg)^{\frac{1}{2}}. \qedhere
\end{equation*}
\end{proof}

\section{Concluding remarks on applications and research directions}

Unlike complex holomorphic function theory, the noncommutative setting of the Clifford algebra allows multiple notions of hyperholomorphicity for vector fields. Consequently, different spectral theories emerge, each based on the concept of spectrum that relies on distinct Cauchy kernels. The spectral theory based on the $S$-spectrum was inspired by quaternionic quantum mechanics (see \cite{BF}), is associated with slice hyperholomorphicity, and began its development in 2006. A comprehensive introduction can be found in \cite{CGK}, with further explorations done in \cite{ACS2016,AlpayColSab2020,ColomboSabadiniStruppa2011}. Applications on fractional powers of operators are investigated in \cite{CGdiffusion2018,CG18,FJBOOK,JONAMEM,JONADIRECT} and some results from classical interpolation theory, see \cite{BERG_INTER,BRUDNYI_INTER,Ale_INTER,TRIEBEL}, have been recently extended into this setting \cite{COLSCH}. In order to fully appreciate the spectral theory on the $S$-spectrum, we recall that it applies to sets of noncommuting operators, for example with the identification $(T_1,...,T_{2^n})\leftrightarrow T=\sum_AT_Ae_A$ it can be seen as a theory for several operators. Extending complex spectral theory to vector operators broadens its applicability to various areas of research fields such as: \medskip

\textit{Quaternionic formulation of quantum mechanics.} The interest in spectral theory for quaternionic operators is motivated by the 1936 paper \cite{BF} on the logic of quantum mechanics by G. Birkhoff, J. von Neumann. There, the authors showed that Schrödinger equation can be written basically in the complex or in the quaternionic setting, see also the book of Adler \cite{adler}. \medskip

\textit{Vector analysis.} The gradient operator with nonconstant coefficients in $n$ dimensions, represents various physical laws, such as Fourier's law for heat propagation and Fick's law for mass transfer diffusion. This can be expressed as
\begin{equation*}
T=\sum_{i=1}^ne_ia_i(x)\partial_{x_i},\qquad x\in\Omega,
\end{equation*}
and is associated with various boundary conditions, as discussed in \cite{CMS24}. Here, $e_1,\dots,e_n$ represent the imaginary units of the Clifford algebra $\mathbb{R}_n$, and the coefficients $a_1,\dots,a_n$ are assumed to belong to $C^1(\overline{\Omega})$ and satisfy appropriate bounds. \medskip

\textit{Differential geometry.} Let $(M,g)$ be a Riemannian manifold of dimension $n$, and $U$ a coordinate neighbourhood of $M$. In this neighbourhood, we can find an orthonormal frame of vector fields $E_1,\dots,E_n$, such that the Dirac operator on $(U,g)$ can be expressed as the $\mathcal{C}^\infty(U,\mathcal{H})$ differential operator
\begin{equation}\label{Eq_Dirac_manifold}
\mathcal{D}=\sum_{i=1}^ne_i\nabla_{E_i}^\tau,
\end{equation}
where $\nabla_{E_i}^\tau$ denotes the covariant derivative with respect to the vector field $E_i$. For more details see the book \cite{DiracHarm}. As a special case the paper \cite{DIRACHYPSPHE} considers the Dirac operator in hyperbolic and spherical spaces, where \eqref{Eq_Dirac_manifold} takes the explicit forms
\begin{equation*}
\mathcal{D}_H=x_n\sum_{i=1}^ne_i\partial_{x_i}-\frac{n-1}{2}e_n\qquad\text{and}\qquad\mathcal{D}_S=(1+|x|^2)\sum_{i=1}^ne_i\partial_{x_i}-nx,
\end{equation*}
where for $\mathcal{D}_H$ we pick $(x_1,...,x_{n-1},x_n)\in\mathbb{R}^{n-1}\times\mathbb{R}^+$ and $(x_1,...,x_n)\in\mathbb{R}^n$ for $\mathcal{D}_S$, respectively. \medskip

\textit{Hypercomplex analysis.} The well-known Dirac operator and its conjugate
\begin{equation*}
D=\partial_{x_0}+\sum_{i=1}^ne_i\partial_{x_i}\qquad\text{and}\qquad\overline{D}=\partial_{x_0}-\sum_{i=1}^ne_i\partial_{x_i},
\end{equation*}
are widely investigated in \cite{DSS} as well as in \cite{DiracHarm}. Slice hyperholomorphic functions can also be viewed as functions in the kernel of the global operator
\begin{equation*}
G=|\Im(x)|^2\partial_{x_0}+\Im(x)\sum_{i=1}^nx_i\partial_{x_i},
\end{equation*}
introduced in \cite{6Global}, where $\Im(x)=x_1e_1+\dots+x_ne_n$. For integers $\alpha,\beta,m$, the operators of the Dirac fine structure on the $S$-spectrum are defined as
\begin{equation*}
T_{\alpha,m}=D^\alpha(D\overline{D})^m\qquad\text{and}\qquad\widetilde{T}_{\beta,m}=\overline{D}^\beta(D\overline{D})^m.
\end{equation*}
The fine structures on the $S$-spectrum constitute a set of function spaces and their related functional calculi. These structures have been introduced and studied in recent works \cite{polypolyFS,CDPS1,Fivedim,Polyf1,Polyf2}, while their respective $H^\infty$-versions are investigated in \cite{MILANJPETER,MPS23}. \medskip

The $H^\infty$-functional calculi also have a broader context in the recently introduced fine structures on the $S$-spectrum. These are function spaces of nonholomorphic functions derived from the Fueter-Sce extension theorem, which connects slice hyperholomorphic and axially monogenic functions, the two main notions of holomorphicity in the Clifford setting. The connection is established through powers of the Laplace operator in a higher dimension, see \cite{Fueter,TaoQian1,Sce} and also the translation \cite{ColSabStrupSce}. \medskip

It is important to highlight that in the hypercomplex setting there exists another spectral theory based on the monogenic spectrum, which was initiated by Jefferies, McIntosh and J. Picton-Warlow in \cite{JM}. This theory is founded on monogenic functions, which are functions in the kernel of the Dirac operator, and their associated Cauchy formula (see \cite{DSS} for a comprehensive treatment). The monogenic spectral theory, based on monogenic functions \cite{DSS}, has been extensively developed and is well described in the seminal books \cite{JBOOK} and \cite{TAOBOOK}. These texts provide an in-depth exploration of the theory, with particular emphasis on the $H^\infty$-functional calculus in the monogenic setting. \medskip

\section*{Declarations and statements}

\textbf{Data availability}. The research in this paper does not imply use of data. \medskip

\textbf{Conflict of interest}. The authors declare that there is no conflict of interest.

\end{document}